\DocumentMetadata{uncompress}
\documentclass[11pt,reqno,letterpaper]{amsart}
\usepackage{amsmath}
\usepackage{amssymb}
\usepackage{amsthm}
\usepackage{empheq}
\usepackage{enumitem}
\usepackage[T1]{fontenc}
\usepackage{mathtools}
\usepackage{microtype}
\usepackage[dvipsnames]{xcolor}
\usepackage{comment}

\usepackage{tikz-cd}
\usepackage{tikz}

\usepackage[%
	style=alphabetic,
	maxnames=5,
	style=alphabetic,
	sorting=anyt,
	backref=true,
	date=year,
	doi=false,
	isbn=false,
]{biblatex}

\renewbibmacro*{volume+number+eid}{%
	\printfield{volume}%
	\setunit*{\addnbthinspace}%
	\printfield{number}%
	\setunit{\addcomma\space}%
\printfield{eid}}
\DeclareFieldFormat[article]{number}{\mkbibparens{#1}}
\DeclareFieldFormat[article]{volume}{\mkbibbold{#1}}
\usepackage{zref-clever}
\usepackage{hyperref}

\zcsetup{cap=true}
\NewDocumentCommand{\mynewtheorem}{momo}{%
	\IfValueTF{#4}
	{\newtheorem{#1}{#3}[#4]}
	{%
		\IfValueTF{#2}
		{%
			\AddToHook{env/#1/begin}{\zcsetup{countertype={#2=#1}}}%
			\newtheorem{#1}[#2]{#3}%
		}
		{\newtheorem{#1}{#3}}%
	}%
}

\hypersetup{
	colorlinks=true,
	linkcolor=Aquamarine,
	citecolor=purple,
	filecolor=cyan,
	urlcolor=teal
}

\theoremstyle{plain}
\mynewtheorem{theorem}{Theorem}[section]
\mynewtheorem{axiom}[theorem]{Axiom}
\mynewtheorem{conjecture}[theorem]{Conjecture}
\mynewtheorem{corollary}[theorem]{Corollary}
\mynewtheorem{hypothesis}[theorem]{Hypothesis}
\mynewtheorem{lemma}[theorem]{Lemma}
\mynewtheorem{proposition}[theorem]{Proposition}

\theoremstyle{definition}
\mynewtheorem{definition}[theorem]{Definition}
\mynewtheorem{assumption}[theorem]{Assumption}
\mynewtheorem{example}[theorem]{Example}
\mynewtheorem{question}[theorem]{Question}
\mynewtheorem{construction}[theorem]{Construction}

\theoremstyle{remark}
\mynewtheorem{claim}[theorem]{Claim}
\mynewtheorem{notation}[theorem]{Notation}
\mynewtheorem{fact}[theorem]{Fact}
\mynewtheorem{remark}[theorem]{Remark}
\mynewtheorem{setup}[theorem]{}
\newtheorem*{remark*}{Remark}
\newtheorem*{notation*}{Notation and Terminology}

\numberwithin{equation}{section}

\DeclareMathOperator{\Ann}{Ann}

\DeclareMathOperator{\Span}{Span}

\DeclareMathOperator{\rk}{rk}

\DeclareMathOperator{\coker}{coker}
\DeclareMathOperator{\Ext}{Ext}
\DeclareMathOperator{\Fit}{Fit}
\DeclareMathOperator{\Hom}{Hom}
\DeclareMathOperator{\SEnd}{\mathcal{E}\!\mathit{nd}}
\DeclareMathOperator{\SExt}{\mathcal{E}\!\mathit{xt}}
\DeclareMathOperator{\SHom}{\mathcal{H}\!\mathit{om}}

\DeclareMathOperator{\id}{id}

\DeclareMathOperator{\Bl}{Bl}

\DeclareMathOperator{\Exc}{Exc}

\DeclareMathOperator{\Gr}{Gr}

\DeclareMathOperator{\Supp}{Supp} 

\DeclareMathOperator{\ch}{ch}

\DeclareMathOperator{\NS}{NS}

\DeclareMathOperator{\CDiv}{CDiv}
\DeclareMathOperator{\Hilb}{Hilb}
\DeclareMathOperator{\Quot}{Quot}

\newcommand{\Romannum}[1]{\uppercase\expandafter{\romannumeral #1}}

\newcommand{\hooklongrightarrow}{\lhook\joinrel\longrightarrow}
\renewcommand{\div}{\operatorname{div}}
\DeclareMathOperator{\red}{red}

\DeclarePairedDelimiter{\abs}{\lvert}{\rvert}
\DeclarePairedDelimiterX{\pair}[2]{\langle}{\rangle}{#1,#2}

\title{Sheaf stable pairs on projective surfaces and birational geometry}

\author{Caucher Birkar}
\address{Yau Mathematical Sciences Center,
	Jingzhai, Tsinghua University, Haidian District,
Beijing, China, 100084}
\email{birkar@tsinghua.edu.cn}

\author{Jia Jia}
\address{Yau Mathematical Sciences Center,
	Jingzhai, Tsinghua University, Haidian District,
Beijing, China, 100084}
\email{jia\_jia@u.nus.edu,mathjiajia@tsinghua.edu.cn}

\author{Artan Sheshmani}
\address{Beijing Institute of Mathematical Sciences and Applications,
	No. 544, Hefangkou Village, Huaibei Town, Huairou District,
Beijing, China, 101408}
\email{artan@bimsa.cn}

\address{Massachusetts Institute of Technology (MIT), IAiFi Institute, 182 Memorial Drive, Cambridge, MA 02139, USA}
\email{artan@mit.edu}

\author{Chengxi Wang}
\address{Yau Mathematical Sciences Center,
	Jingzhai, Tsinghua University, Haidian District,
Beijing, China, 100084}
\email{chxwang@tsinghua.edu.cn}

\date{\today}


\begin{document}

\begin{abstract}
	We study moduli space of higher rank marginally stable pairs
	\(\mathcal{E},s\coloneqq (s_1,\cdots,s_r)\) consisting of torsion free coherent sheaf \(\mathcal{E}\) of rank \(r\)
	and \(r\) sections \(s_1, \dots, s_r\) on a smooth projective surface.
	Having fixed the Chern character of \(\mathcal{E}\), the resulting moduli space is isomorphic to some subscheme of
	the Quot-scheme parametrising quotient sheaves of appropriate Chern character.
	We establish a connection between moduli space of higher rank stable pairs and  stable minimal models
	induced by the sheaf \(\mathcal{E}\) and sections \(s_i\) and the relative lc model of base surface,
	and use birational geometry of minimal models to analyse in detail the components of the fibre of the
	Hilbert-Chow morphism from the moduli space to the Hilbert scheme of effective Cartier divisors on the
	base surface.
\end{abstract}

\maketitle

\setcounter{tocdepth}{1}
\tableofcontents

\section{Introduction}

The classification of geometric objects is a central theme in algebraic geometry,
and moduli theory provides a systematic framework for addressing such classification problems,
offering solutions to the geometric classification problems.
Early in the 1960s, Seshadri introduces the notion of \(S\)-equivalence and constructs a projective moduli
space for semi-stable vector bundles on a smooth curve \cite{Seshadri1967},
which gives a compactification for Mumford's moduli space of stable bundles \cite{Mumford1962}.
Around the 70s of the 20th century,
the moduli spaces of sheaves on higher dimensional varieties have been constructed by
Gieseker~\cite{Gieseker1977} for surfaces and Maruyama~\cite{Maruyama1977,Maruyama1978} in general.
The moduli space of decorated sheaves, i.e., sheaves with additional structures or data attached to them,
plays a significant role in modern algebraic geometry.
As a realisation of the stable pairs in the sense of Thaddeus \cite{Thaddeus1994},
Le Potier introduces a type of decoration, \emph{coherent systems}, and establishes the moduli spaces under an
appropriate stability condition \cite{lepotier1993systemes}.
A coherent system is a coherent sheaf on a projective variety decorated with a subspace inside the space of
global sections of that sheaf.
Extensive applications of coherent systems have been made in Brill-Noether theory \cite{Brambila-Paz2008}.
As for higher dimensional contexts, coherent systems are used in surface cases,
as well as in enumerative geometry such as curve counting on Calabi-Yau threefolds
by Pandharipande and Thomas \cite{pandharipande2009curve}, where they were called \emph{stable pairs}.
The new stable pair invariants are in close relation to Gromov-Witten, Donaldson-Thomas and physicists' BPS invariants.
In this paper, we study the moduli space of a kind of coherent systems which are called \emph{sheaf stable pairs}
\zcref{def:sheaf_stable_pairs},
which is introduced by \cite{sheshmani2016higher} as a higher rank analogue of the theory of stable pairs in
\cite{pandharipande2009curve}.

\smallskip

The points of a moduli space represent geometric objects with some fixed invariants,
or isomorphism classes of some fixed kind.
Analysing the geometric structure of moduli spaces produces effective tools and methods to elucidate the
difference and connection between geometric objects.
In \cite{lepotier1993faisceaux}, Le Potier explored the moduli spaces of semi-stable sheaves on the projective
plane and on the projective space \(\mathbb{P}^3\) with fixed Chern classes or fixed Hilbert polynomial.
He provided geometric description of the Barth morphisms from moduli space of semi-stable sheaves to projective
space parametrising quartic curves, and further studied the irreducible components of the moduli space.
An often instrumental approach in studying the moduli spaces is by connecting them to other parameter spaces.

\smallskip

Varieties of general type, Calabi-Yau varieties and Fano varieties are three fundamental building blocks in
complex algebraic varieties.
In the language of differential geometry, they roughly correspond to negative,
zero and positive curvature spaces.
The moduli theories for varieties of
general type and for Fano varieties have been established and relatively well-understood.

\smallskip

Early in the 1960s, the moduli spaces of smooth projective curves of fixed genus \(g\geq 2\) were compactified
by introducing stable nodal curves as good limits \cite{DeligneMumford1969} to get the coarse moduli space
\(\overline{M_g}\) of Deligne-Mumford stable curves.
The theory was extended to surfaces of general type in \cite{KollarShepherd-Barron1988}.
Then Alexeev \cite{Alexeev1996} constructed analogues of \(M_{g,n}\) and \(M_{g,n}(W)\) for surfaces,
where the \(M_{g,n}\) parametrize stable \(n\)-pointed curves and
\(M_{g,n}(W)\) is a moduli of stable maps from reduced curves to a variety \(W\)
studied by Kontsevich in enumerative geometry of curves \cite{Kontsevich1995}.
Starting from pairs \((X,B)\) with surface \(X\) and a divisor \(B\) such that \(K_X+B\) is ample,
and then applying Log Minimal Model Program in dimension \(3\) and semi-stable reduction,
one may get semi-log canonical singularities.
Then the so-called KSBA-stable pairs of general type, as well as stable maps of pairs,
are defined \cite[Definition~2.1.]{Alexeev1996}.
A KSBA-stable pair of general type consists of connected projective surface \(X\) and a divisor \(B\) with
reduced but not necessarily irreducible components such that \((X,B)\) has semi-log canonical singularities and
\(K_X+B\) is ample.
Heavily based on progress in birational geometry, especially, on existence of minimal models
\cite{BirkarCasciniHaconMcKernan2010} and boundedness of varieties of general type \cite{HaconMcKernanXu2018},
this moduli theory has been generalized to higher dimension completely due to a lot of people
(see \cite{kollar2022,Alexeev96}).

\smallskip

In \cite{birkar2022moduli}, the first author defines the stable varieties of arbitrary non-negative Kodaira
dimension that generalises both KSBA-stable varieties of general type and polarised Calabi-Yau varieties,
establish their boundedness under natural assumptions, and then construct their projective coarse moduli spaces.
The stable objects defined in \cite{birkar2022moduli} are called \emph{stable minimal models}
(see \zcref{def:lc_stable_minimal_model}) which contain various kinds of geometric
objects since they involve stable varieties of arbitrary non-negative Kodaira dimensions.

\medskip

It is therefore natural to seek relations between moduli spaces parametrising different types of geometric objects.
In \cite{birkar2024sheaf}, a bridge is established between stable sheaf pairs and stable minimal models.
In the present paper we study this relation in detail on smooth projective surfaces.
The explicit construction of this bridge involves birational operations such as running the minimal model program,
taking log canonical models, and performing Maruyama’s elementary transformations (\zcref{sec:geometric_treatment}).
This allows us to analyse the global geometry of the moduli space of sheaf stable pairs,
describe its components, interpret it as a fibre bundle, and study its fibres
(\zcref{sec:pic_one_surfaces,sec:negative_curves}).

Since sheaf stable pairs supported on surfaces have been extensively investigated in enumerative geometry,
the results of this paper suggest further connections between enumerative and birational geometry.
Our work provides a new perspective and lays the groundwork for a series of subsequent developments.

\medskip

We now recall the basic notions and state the main results.
We begin with the definition of sheaf stable pairs.
This notion arises as a special case of marginal stability for a parameter-dependent stability condition
associated to a sheaf pair (cf.~\cite[Section~3]{birkar2024sheaf}).

\begin{definition}[{\cite[Definitions~1.1, 3.5]{birkar2024sheaf}}]\label{def:sheaf_stable_pairs}
	Let \(Z\) be an algebraic variety.
	A \emph{sheaf stable pair} \(\mathcal{E},s\) on \(Z\) consists of a torsion-free sheaf \(\mathcal{E}\) of rank \(r>0\)
	and a morphism of \(\mathcal{O}_{Z}\)-modules \(\mathcal{O}_{Z}^r\xrightarrow{s}\mathcal{E}\)
	such that \(\dim\Supp\coker(s)<\dim Z\).
\end{definition}

Two stable pairs \(\mathcal{E},s\) and \(\mathcal{G},t\) are said to be equivalent if there exists a commutative diagram
\[
	\begin{tikzcd}
		\mathcal{O}_{Z}^r \ar[r,"s"] \ar[d,equal] & \mathcal{E} \ar[d,"\cong"]\\
		\mathcal{O}_{Z}^r \ar[r,"t"'] & \mathcal{G}
	\end{tikzcd}
\]
The equivalence class of \(\mathcal{E},s\) is denoted by \([\mathcal{E},s]\).

For a smooth projective variety \(Z\),
we denote by \(M_Z(\ch)\) the moduli space of stable pair classes \([\mathcal{E},s]\)
with Chern character \(\ch(\mathcal{E})=\ch\) on \(Z\).
It is known that \(M_Z(\ch)\) is a projective scheme (\cite{lepotier1993faisceauxb} and
\cite[Lemma~3.7]{sheshmani2016higher}).

\smallskip

Fix a smooth projective variety \(Z\).
Let \(\CDiv(Z)\) be the subscheme of the Hilbert scheme \(\Hilb(Z)\) parametrizing effective Cartier divisors on \(Z\)
(cf.~\cite[Chapter I, Definition~1.12 and Theorem~1.13]{kollar1996rational}).
There is a set-theoretic map
\[
	M_Z(\ch)\longrightarrow\CDiv(Z)
\]
sending a sheaf stable pair \(\mathcal{E},s\) to the effective  Cartier divisor \(\mathcal{Z}_f(\coker s)\)
(see \zcref{def:annihilator_support,cor:fitting_support_cartier}).
When this map is a well-defined morphism of projective schemes,
we can study the geometry of \(M_Z(\ch)\) via the fibres of this Hilbert-Chow morphism.

\begin{remark}
	Let \(Z\) be a smooth projective surface of Picard number \(1\).
	Let \(H\) be an ample generator of the N\'eron-Severi group of \(Z\).
	For any Cartier divisor \(D\) on \(Z\) we define its \emph{degree} by
	\[
		\deg(D)\coloneqq D\cdot H.
	\]
	In particular, for any \(d\geq 1\),
	the complete linear system \(\abs{dH}\) consists of all effective divisors of degree \(dH^2\).
	Therefore, \(\CDiv^d(Z)=\mathbb{P}(H^{0}(Z,\mathcal{O}_{Z}(dH)))\) is smooth.

	By contrast, let \(Z'\) be another smooth projective surface
	and let \(C\subseteq Z'\) be an irreducible curve.
	Fix a very ample divisor \(H'\) and let
	\[
		P(m)=\chi(\mathcal{O}_{nC}(mH'))
	\]
	be the Hilbert polynomial of the structure sheaf of the multiple \(nC\) with respect to \(H'\).
	Here \(nC\) is the subscheme of \(Z'\) with ideal sheaf \(\mathcal{O}_{Z'}(-nC)\).
	If \(C\) is a \emph{negative curve} (i.e., \(C^2<0\)),
	then the normal bundle of the multiple \(nC\) satisfies
	\[
		H^0(Z',\mathcal{N}_{nC/Z'})=0,
	\]
	so the Zariski tangent space to the scheme \(\CDiv^P(Z')\) at \([nC]\) is trivial.
	Consequently, the irreducible component of \(\CDiv^P(Z')\) containing \([nC]\) consists of a single point.
\end{remark}

In this paper we mainly consider sheaf stable pairs on smooth projective surfaces of Picard number \(1\)
in \zcref{sec:pic_one_surfaces} and
sheaf stable pairs whose first Chern class is a multiple of a negative rational curve
on a smooth projective surface in \zcref{sec:negative_curves}.

\medskip

We first describe the moduli space in the Picard number one case by analysing the fibres of the Hilbert–Chow morphism.

\begin{theorem}\label{thm:moduli_space_picard_one}
	Let \(Z\) be a smooth projective surface of Picard number one
	and let \(H\) be an ample generator of \(\NS(Z)\).
	Consider an element \([\mathcal{E},s]\in M_Z(r,dH,d^2H^2/2)\),
	i.e., a sheaf stable pair \(\mathcal{E},s\) with rank \(\ch(\mathcal{E})=(r,dH,d^2H^2/2)\).
	Then \(\mathcal{E}\) is locally free.
	Moreover, there exists a Hilbert-Chow morphism
	\[
		h\colon M_Z(r,dH,d^2H^2/2)\longrightarrow \CDiv^d(Z)
	\]
	whose fiber over any point \([n_iC_i]\) is isomorphic to \(\mathbb{P}^{r-1}\).
\end{theorem}

In what follows, let \(Z\) be a smooth projective surface and
let \(C\) be a rational curve on \(Z\) with \(C^2=-d\) for some \(d>0\).
We then describe the moduli space of sheaf stable pairs as follows.

\begin{theorem}[=\zcref{thm:moduli_space_d1}]
	If \(d=1\), the moduli space \(M_Z(2,2C,-2)\) is isomorphic to the \(\mathbb{P}^2\)-bundle
	\(\mathbb{P}_{\mathbb{P}^1}(\mathcal{O}(2)^{\oplus 2}\oplus\mathcal{O})\).
\end{theorem}

\begin{theorem}[=\zcref{thm:main_negative_curve}]
	For \(d\geq 2\), the moduli space \(M_Z(2,2C,-2d)\) has two smooth irreducible components:
	one is isomorphic to \(M_1=\mathbb{P}_{\mathbb{P}^1}(\mathcal{O}(2)^{\oplus d+1}\oplus\mathcal{O})\),
	another is isomorphic to \(M_2=\Quot(\mathcal{O}_{C}^2,d)\).
	Moreover, \(M_1\) and \(M_2\) intersects along a copy of \(\mathbb{P}^d\times\mathbb{P}^1\),
	which is the section at infinity of the \(\mathbb{P}^{d+1}\)-bundle \(M_1\to\mathbb{P}^1\).
\end{theorem}

\subsection*{Acknowledgements}
The first author is supported by a grant from Tsinghua University and
a grant of the National Program of Overseas High Level Talent.
The second author is supported by Shuimu Tsinghua Scholar Program and
China Postdoctoral Science Foundation (2023TQ0172).
The third author is supported by grants from Beijing Institute of Mathematical Sciences and Applications (BIMSA),
the Beijing NSF BJNSF-IS24005,
and the China National Science Foundation (NSFC) NSFC-RFIS program W2432008.
He would also like to thank China's National Program of Overseas High Level Talent for generous support.
Finally, he would like to thank NSF AI Institute for Artificial Intelligence and Fundamental Interactions at
Massachusetts Institute of Technology (MIT) which is funded by the US NSF grant under Cooperative Agreement
PHY-2019786.
\section{Preliminaries}

We work over an algebraically closed field \(\mathbf{k}\) of characteristic zero.
All varieties and schemes are defined over \(\mathbf{k}\) unless stated otherwise, and
varieties are assumed to be irreducible.

\subsection{Contractions}

A \emph{contraction} is a projective morphism \(f\colon X\to Y\) of schemes such that \(f_*\mathcal{O}_X=\mathcal{O}_Y\)
In particular, \(f\) is surjective and has connected fibres.

\subsection{Pairs and singularities}

Let \(X\) be a pure-dimensional scheme of finite type over \(\mathbf{k}\),
and let \(B\) be a \(\mathbb{Q}\)-divisor on \(X\).
We denote the coefficient of a prime divisor \(D\) in \(B\) by \(\mu_D B\).

A \emph{pair} \((X,B)\) consists of a normal quasi-projective variety \(X\)
and a \(\mathbb{Q}\)-divisor \(B\geq 0\) such that \(K_X+B\) is \(\mathbb{Q}\)-Cartier.
We call \(B\) the \emph{boundary divisor}.

Let \(\phi\colon W\to X\) be a log resolution of the pair \((X,B)\).
Let \(K_W+B_W\) be the pullback of \(K_X+B\).
The \emph{log discrepancy} of a prime divisor \(D\) on \(W\) with respect to \((X,B)\) is defined as
\[
	a(D,X,B)\coloneqq 1-\mu_D B_W.
\]
We say that \((X,B)\) is lc (resp.\ klt) if \(a(D,X,B)\) is \(\geq 0\) (resp.\ \(>0\)) for every \(D\);
equivalently, every coefficient of \(B_W\) is \(\leq 1\) (resp.\ \(<1\)).

A \emph{log smooth} pair is a pair \((X,B)\) with \(X\) smooth and \(\Supp B\) having simple normal crossing
singularities.

\subsection{Types of models}

Let \((X,B)\) be an lc pair.

Let \(X\to Z\) be a contraction to a normal variety and
assume that \(K_X+B\) is big over \(Z\).
We say that the \emph{log canonical model} of \((X,B)\) over \(Z\) exists
if there is a birational contraction \(\varphi\colon X\dashrightarrow Y\) where
\begin{itemize}
	\item \(Y\) is normal and projective over \(Z\);
	\item \(K_Y+B_Y\coloneqq \varphi_*(K_X+B)\) is ample over \(Z\); and
	\item \(\alpha^*(K_X+B)\geq \beta^*(K_Y+B_Y)\) for any common resolution
		\[
			\begin{tikzcd}
				& W \ar[ld,swap,"\alpha"] \ar[rd,"\beta"] \\
				X \ar[rr,dashrightarrow,swap,"\varphi"] \ar[rd] & & Y \ar[ld] \\
				& Z.
			\end{tikzcd}
		\]
\end{itemize}
We call \((Y,B_Y)\) the log canonical model of \((X,B)\) over \(Z\).

\medskip

Now we recall the definition of stable minimal model in relative setting used in \cite{birkar2024sheaf}.

\begin{definition}[{\cite[Definition~4.3]{birkar2024sheaf}}]\label{def:lc_stable_minimal_model}
	Let \(Z\) be an algebraic variety.
	A log canonical \emph{stable minimal model} over \(Z\) is of the form
	\[
		(X,B), A \xlongrightarrow{f} Z
	\]
	where
	\[
		\begin{cases}
			(X,B) \text{ is a log canonical pair equipped with a projective morphism \(X \xlongrightarrow{f} Z\)}, \\
			K_X+B \text{ is semi-ample over \(Z\)},                                                                \\
			A\geq 0 \text{ is an integral divisor on } X,                                                        \\
			K_X+B+uA \text{ is ample over } Z \text{ for } 0<u\ll 1,                                             \\
			(X,B+uA) \text{ is log canonical for } 0<u\ll 1.
		\end{cases}
	\]
\end{definition}

\subsection{Chern character}

We will study moduli spaces of sheaf stable pairs with fixed Chern character.
The following lemma is convenient for computing Chern characters of sheaves supported on divisors.

\begin{lemma}[{\cite[p.30, Lemma~1]{friedman1998algebraic}}]
	Let \(X\) be a smooth quasi-projective variety and let \(j\colon D\hookrightarrow X\) be an effective Cartier divisor.
	Suppose that \(\mathcal{L}\) is a line bundle on \(D\).
	Then
	\[
		\ch_1(j_*\mathcal{L})=[D], \quad
		\ch_2(j_*\mathcal{L})=-\frac{1}{2}[D]^2+j_*\ch_1(\mathcal{L}).
	\]
\end{lemma}

\subsection{Support of a sheaf}

Let \(X\) be a Noetherian scheme and \(\mathcal{F}\) a coherent sheaf on \(X\).
The support of \(\mathcal{F}\) is the closed set \(\Supp(\mathcal{F})\coloneqq\{x\in X\mid\mathcal{F}_x\neq 0\}\).
Its dimension is called the dimension of the sheaf \(\mathcal{F}\).

\begin{definition}[{cf.~\cite[Section~1.1]{huybrechts2010geometry}}]\label{def:annihilator_support}
	The \emph{annihilator support} of \(\mathcal{F}\) is the closed subscheme \(\mathcal{Z}_a(\mathcal{F})\)
	defined by the annihilator ideal sheaf \(\Ann(\mathcal{F})\) of \(\mathcal{F}\),
	which is the kernel of \(\mathcal{O}_{X}\to \SEnd(\mathcal{F})\).
\end{definition}

The annihilator support of \(\mathcal{F}\) is the smallest closed subscheme \(i\colon Z\hookrightarrow X\)
such that the canonical map \(\mathcal{F}\to i_*i^{*}\mathcal{F}\) is an isomorphism.

\begin{definition}[{\cite[Definition~1.1.2]{huybrechts2010geometry}}]
	A coherent sheaf \(\mathcal{F}\) is \emph{pure} of dimension \(d\)
	if \(\dim\Supp(\mathcal{F}')=d\) for all non-zero coherent subsheaves \(\mathcal{F}'\subseteq \mathcal{F}\).
\end{definition}

We next introduce Fitting support, which is the appropriate notion for the Hilbert–Chow morphism.

\begin{definition}[{\cite[Tag~0C3C]{stacks-project}}]
	The \emph{Fitting support} of \(\mathcal{F}\) is the closed subscheme \(\mathcal{Z}_f(\mathcal{F})\)
	defined by the Fitting ideal \(\operatorname{Fit}_0(\mathcal{F})\) of \(\mathcal{F}\),
	which is constructed from local  presentations of \(\mathcal{F}\).
	Namely, if \(U\subseteq X\) is open, and
	\[
		\mathcal{O}_{U}^m\longrightarrow \mathcal{O}_{U}^n\longrightarrow \mathcal{F}|_{U}\longrightarrow 0
	\]
	is a presentation of \(\mathcal{F}\) over \(U\) with \(m\geq n\),
	then \(\Fit_0(\mathcal{F})|_U\) is generated by the \(n\times n\)-minors
	of the matrix defining the presentation.
\end{definition}

\begin{remark}[{\cite[Tag~0CYX]{stacks-project}}]
	With the notation above, we have:
	\begin{itemize}
		\item \(\Supp(\mathcal{F})\subseteq\mathcal{Z}_a(\mathcal{F})\subseteq \mathcal{Z}_f(\mathcal{F})\)
			as closed subschemes;
		\item \(\Supp(\mathcal{F})=\mathcal{Z}_a(\mathcal{F})=\mathcal{Z}_f(\mathcal{F})\) as closed subsets;
		\item when \(\mathcal{F}\) is locally principal, one has \(\mathcal{Z}_a(\mathcal{F})=\mathcal{Z}_f(\mathcal{F})\).
	\end{itemize}
\end{remark}
Here, locally principal means that the sheaf is locally generated by a single element.

\begin{remark}\label{rmk:divisor_fitting_support}
	Let \(X\) be a smooth variety and let \(\mathcal{F}\) be a pure sheaf on \(X\) with codimension \(1\).
	Following \cite[Remark~7.1.5]{fantechi2006fundamental},
	for each prime divisor \(D\) on \(X\),
	there is an open subset \(U\subseteq X\) with \(U\cap D\neq\emptyset\),
	on which we have a resolution
	\[
		0\longrightarrow \mathcal{O}_U^n\xlongrightarrow{\varphi} \mathcal{O}_U^n\longrightarrow
		\mathcal{F}|_U\longrightarrow 0.
	\]
	Let \(m_D\) be the vanishing order of \(\det(\varphi)\) on an open subset of \(D\).
	Then the Fitting support of \(\mathcal{F}\) is the effective Cartier divisor
	\[
		\div(\mathcal{F})\coloneqq \sum_D m_DD.
	\]
	Moreover, \(\ch_1(\mathcal{F})=\div(\mathcal{F})\) (cf.~\cite[Section~5.3]{mumford1994geometric}).
\end{remark}

\begin{example}\label{exa:fitting_support_determinant}
	Let \(X\) be a normal algebraic variety and \(\mathcal{E}\) a locally free sheaf of rank \(r\).
	Consider a sheaf stable pair
	\[
		0\longrightarrow \mathcal{O}_{X}^r\longrightarrow \mathcal{E}\longrightarrow \mathcal{Q}\longrightarrow 0
	\]
	given by global sections \(s_{1},\dots,s_{r}\in H^0(X,\mathcal{E})\).
	The Fitting support of \(\mathcal{Q}\) is defined by the vanishing of
	\(s_1\wedge\dots\wedge s_r\in H^0(X,\det\mathcal{E})\).
	This is the \emph{degeneracy locus} of the sections \(s_1,\dots,s_r\),
	which coincides as a cycle with the first Chern class of \(\mathcal{E}\)
	(see \cite[Example~14.4.2]{fulton1998intersection}).
\end{example}

\begin{example}
	Let \(i\colon D\hookrightarrow X\) be an effective Cartier divisor on a smooth projective variety.
	Note that the ideal sheaf \(\mathcal{O}_{X}(-D)\) is locally principal.
	If \(\mathcal{F}\) is a locally free sheaf of rank \(r\) on \(D\),
	then the Fitting support of \(i_*\mathcal{F}\) is defined by the ideal sheaf \(\mathcal{O}_{X}(-rD)\),
	the \(r\)-th power of \(\mathcal{O}_{X}(-D)\).
	In this case, \(\ch_1(i_*\mathcal{F})=rD\).
\end{example}

\subsection{Cross-sections of ruled surfaces}

The results of this subsection will be used in \zcref{sec:geometric_treatment}.

Let \(Z\) be a smooth projective surface and let \(X\to Z\) be a \(\mathbb{P}^1\)-bundle.
For any subvariety \(L\) of \(Z\), set \(X_L\coloneqq \pi^{-1}(L)\);
then \(X_L\to L\) is again a \(\mathbb{P}^1\)-bundle.

\begin{proposition}\label{prop:self_intersection_blowup_curve}
	Suppose that \(L\subseteq Z\) is a smooth irreducible curve with self-intersection \(b\in\mathbb{Z}\).
	Let \(\ell\) be a cross-section of the induced \(\mathbb{P}^1\)-bundle
	\[
		S\coloneqq X_L\to L.
	\]
	Let \(\sigma\colon\widetilde{X}\to X\) be the blowup along \(\ell\), with exceptional divisor \(E\),
	and let \(\widetilde{S}\) be the strict transform of \(S\).
	Set \(C\coloneqq E\cap \widetilde{S}\).
	If \((\ell^2)_S=a\), then \((C^2)_E=b-a\).
\end{proposition}

\begin{proof}
	Note that \(E\simeq\mathbb{P}(\mathcal{N}_{\ell/X})\to \ell\) is a \(\mathbb{P}^1\)-bundle
	and \(C\) is a cross-section.
	By the projection formula \(S\cdot\ell=L^2=b\).
	The normal bundles sequence
	\[
		0\longrightarrow \mathcal{N}_{\ell/S}\simeq \mathcal{O}_{\ell}(a)\longrightarrow
		\mathcal{N}_{\ell/X}\longrightarrow
		\mathcal{N}_{S/X}|_{\ell}\simeq \mathcal{O}_{X}(S)|_{\ell}\simeq \mathcal{O}_{\ell}(b)\longrightarrow 0,
	\]
	yields \(\deg\mathcal{N}_{\ell/X}=a+b\).
	Since \(\widetilde{S}=\sigma^{*}S-E\),
	\begin{align*}
		(C^2)_E & = (\widetilde{S}|_E)^2_E = \widetilde{S}^2\cdot E \\
		& = (\sigma^{*}S-E)^2\cdot E                            \\
		& = \sigma^{*}S^2\cdot E - 2\sigma^{*}S\cdot E^2 + E^3  \\
		& = -2S\cdot\sigma_*(E^2) + E^3                         \\
		& = 2S\cdot\ell - \deg \mathcal{N}_{\ell/X}             \\
		& = b-a.
	\end{align*}
	This proves the claim.
\end{proof}

We next determine the \emph{invariant} of the ruled surface \(E\to\ell\) in some circumstance
\cite[Chapter~V, Proposition~2.8]{hartshorne1977algebraic}.

\begin{corollary}\label{cor:ruled_surface_inv}
	Under the assumption of \zcref{prop:self_intersection_blowup_curve},
	suppose that there exists a second cross-section \(C'\) of \(E\to \ell\) disjoint from \(C\),
	then \(E\) is a ruled surface over \(\ell\) with invariant \(\abs{a-b}\).
\end{corollary}

\begin{proof}
	The existence of two disjoint cross‑sections forces the invariant of \(E\) to be non‑negative.
	From the normal bundle sequence we obtain
	\[
		0\longrightarrow \mathcal{O}_{\ell}\longrightarrow \mathcal{N}_{\ell/X}(-a)\longrightarrow
		\mathcal{O}_{\ell}(b-a)\longrightarrow 0.
	\]
	If \(b-a\leq 0\), then \(\mathcal{N}_{\ell/X}(-a)\) is normalised
	and the invariant of \(E\simeq \mathbb{P}(\mathcal{N}_{\ell/X})\) equals \(a-b\).

	If \(b-a>0\), then \((C^2)_E=b-a>0\) and \(C\cdot C'=0\).
	By the Hodge Index Theorem \cite[Chapter~V, Theorem~1.9]{hartshorne1977algebraic}, \((C'^2)_E<0\).
	Thus \(C'\) is the unique irreducible negative curve on \(E\),
	and the invariant of \(E\to\ell\) equals \(b-a\).
\end{proof}

Let \(\mathbb{F}_n\) denote the Hirzebruch surface of degree \(n\geq 1\),
and let \(C_0\) be its negative section.
We write
\[
	D(\mathbb{F}_n\to\mathbb{P}^1,C_0)=
	\{\text{cross-sections of } \mathbb{F}_n\to \mathbb{P}^1 \text{ disjoint from } C_0\},
\]
and for a point \(p\not\in C_0\) we set
\[
	D(\mathbb{F}_n\to \mathbb{P}^1,C_0,p)=
	\{\text{cross-sections of } \mathbb{F}_n\to \mathbb{P}^1 \text{disjoint from } C_0 \text{ and passing through } p\}.
\]

\begin{lemma}\label{lem:parameter_space_cross_sections}
	The space \(D(\mathbb{F}_n\to\mathbb{P}^1,C_0)\) is parametrised by \(\mathbb{A}^{n+1}\),
	and the subspace \(D(\mathbb{F}_n\to\mathbb{P}^1,C_0,p)\) is parametrised by \(\mathbb{A}^n\).
\end{lemma}

\begin{proof}
	Let \(F\) be a general fibre of \(\mathbb{F}_n\to\mathbb{P}^1\).
	Any section \(C\neq C_0\) of \(\mathbb{F}_n\to\mathbb{P}^1\) lies in the linear system \(\abs{C_0+mF}\)
	for some \(m\geq n\) (\cite[Chapter~V, Theorem 2.17]{hartshorne1977algebraic}).
	Since
	\[
		C\cdot C_0 = (C_0+mF)\cdot C_0 = -n + m,
	\]
	the condition that \(C\) be disjoint from \(C_0\) forces \(m=n\).
	Hence every desired cross‑section lies in the complete linear system \(\abs{C_0+nF}\).

	The reducible members of \(\abs{C_0+nF}\) are precisely the curves \(C_0+D'\) with \(D'\in\abs{nF}\).
	Since \(h^0(\mathbb{F}_n,C_0+nF)=n+2\) and \(h^0(\mathbb{F}_n,nF)=n+1\),
	Consequently
	\[
		D(\mathbb{F}_n\to\mathbb{P}^1,C_0)\simeq \abs{C_0+nF}\setminus\abs{nF}\simeq\mathbb{A}^{n+1}.
	\]
	The linear systems \(\abs{C_0+nF}\) and \(\abs{nF}\) are both base-point-free
	\cite[Chapter~V, Theorem~2.17]{hartshorne1977algebraic}.
	Fixing a point \(p\not\in C_0\) cuts each linear system by a hyperplane, so that
	\[
		\abs{C_0+nF-p}\simeq \mathbb{P}^n, \quad \abs{nF-p}\simeq \mathbb{P}^{n-1}.
	\]
	Their difference is an affine space of dimension \(n\), giving
	\(D(\mathbb{F}_n\to\mathbb{P}^1,C_0,p)\simeq\mathbb{A}^n\).
\end{proof}

\section{Sheaf stable pairs revisit}

Let \(Z\) be an algebraic variety and let \(\mathcal{E},s\) be a \emph{sheaf stable pair} of rank \(r\) on \(Z\)
(cf.~\zcref{def:sheaf_stable_pairs}).
The morphism \(s\colon \mathcal{O}_{Z}^r\to \mathcal{E}\) is determined by \(r\) global sections
\(s_{1}, \dots, s_{r}\in H^{0}(Z,\mathcal{E})\).

Let \(\mathcal{Q}\coloneqq\coker(s)\) be the cokernel sheaf of the sheaf stable pair.
The \emph{cokernel subscheme} of \((\mathcal{E},s)\) is the Fitting support
\[
	Q\coloneqq \mathcal{Z}_f(\mathcal{Q}).
\]
If \(\mathcal{E}\) is locally free, then the global section
\[
	s_1\wedge s_2\wedge \dots\wedge s_r\in H^0(Z,\det \mathcal{E})
\]
vanishes exactly along the locus where the \(r\) sections become linearly dependent.
In that case \(Q\) is the zero locus of the determinant section \(s_1\wedge s_2\wedge\dots\wedge s_r\)
(cf.~\zcref{exa:fitting_support_determinant}).

\begin{lemma}\label{lem:purity_of_cokernel}
	Let \(Z\) be a smooth projective variety of dimension \(n\).
	Suppose that \(\mathcal{E},s\) is a sheaf stable pair on \(Z\) with non-zero cokernel sheaf \(\mathcal{Q}\).
	Then \(\mathcal{Q}\) is a pure sheaf of dimension \(n-1\).
\end{lemma}

\begin{proof}
	Let \(\omega_Z\) be the dualising sheaf of \(Z\).
	From the short exact sequence
	\[
		0\longrightarrow \mathcal{O}_Z^r\xlongrightarrow{s} \mathcal{E}\longrightarrow \mathcal{Q}\longrightarrow 0,
	\]
	we obtain, by applying \(\SHom(-,\omega_Z)\)
	\[
		0\longrightarrow \mathcal{E}^D\longrightarrow \mathcal{O}_{Z}^r\longrightarrow
		\SExt^1(\mathcal{Q},\omega_Z)\longrightarrow \SExt^1(\mathcal{E},\omega_Z)\longrightarrow 0
	\]
	where \(\mathcal{E}^D\coloneqq \SHom(\mathcal{E},\omega_Z)\) (\cite[Definition~1.1.7]{huybrechts2010geometry}).
	Let \(\mathcal{F}\subseteq\SExt^1(\mathcal{Q},\omega_Z)\) be the image of \(\mathcal{O}_{Z}^r\).
	Then \(\dim\Supp(\mathcal{F})\leq n-1\), and the above sequence splits as
	\[
		0\longrightarrow \mathcal{E}^D\longrightarrow\mathcal{O}_{Z}^r\longrightarrow \mathcal{F}\longrightarrow 0.
	\]
	Applying \(\SHom(-, \omega_Z)\) again, we get
	\[
		0\longrightarrow \mathcal{O}_{Z}^r\longrightarrow \mathcal{E}^{DD}\longrightarrow
		\SExt^1(\mathcal{F},\omega_Z)\longrightarrow 0.
	\]
	The natural map from \(\theta_E\colon \mathcal{E}\to\mathcal{E}^{DD}\) is injective
	because \(\mathcal{E}\) is torsion-free (\cite[Proposition~1.1.10]{huybrechts2010geometry}).
	We obtain a commutative diagram with exact rows
	\[
		\begin{tikzcd}
			0 \ar[r] & \mathcal{O}_Z^r \ar[r,"s"] \ar[d, "\cong"] & \mathcal{E} \ar[r] \ar[d, "\theta_{\mathcal{E}}"]
			& \mathcal{Q}
			\ar[r] \ar[d] & 0 \\
			0 \ar[r] & \mathcal{O}_{Z}^r \ar[r,"s^{DD}"] & \mathcal{E}^{DD} \ar[r] & \SExt^1(\mathcal{F},\omega_Z) \ar[r] & 0.
		\end{tikzcd}
	\]
	By the snake lemma we get an injection
	\[
		\mathcal{Q}\hooklongrightarrow \SExt^1(\mathcal{F},\omega_Z).
	\]
	Hence \(\mathcal{Q}\) can be viewed as a subsheaf of the (torsion) sheaf \(\SExt^1(\mathcal{F},\omega_Z)\).
	By assumption, \(\dim\Supp(\mathcal{Q})=n-1\) (see also \cite[Lemma~3.7]{birkar2024sheaf}),
	so \(\SExt^1(\mathcal{F},\omega_Z)\) is a reflexive sheaf of dimension \(n-1\)
	(\cite[Proposition~1.1.10]{huybrechts2010geometry}).

	Any subsheaf of a reflexive sheaf of pure dimension \(n-1\) is itself pure of the same dimension.
	Therefore \(\mathcal{Q}\) is pure of dimension \(n-1\), as claimed.
\end{proof}

The following corollary is a direct consequence of \zcref{lem:purity_of_cokernel,rmk:divisor_fitting_support}.

\begin{corollary}\label{cor:fitting_support_cartier}
	Let \(Z\) be a smooth projective variety.
	Suppose that \(\mathcal{E},s\) is a sheaf stable pair with cokernel sheaf \(\mathcal{Q}\).
	Then the Fitting support of \(\mathcal{Q}\) forms an effective Cartier divisor.
	Moreover, \(\mathcal{Z}_f(\mathcal{Q})=\ch_1(\mathcal{E})\).
\end{corollary}

\begin{proof}
	By \zcref{lem:purity_of_cokernel}
	the sheaf \(\mathcal{Q}\) is pure of codimension \(1\).
	For a pure sheaf of codimension \(1\) its Fitting support is an effective Cartier divisor;
	see \zcref{rmk:divisor_fitting_support}.

	From the short exact sequence
	\[
		0\longrightarrow \mathcal{O}_{Z}^r\longrightarrow \mathcal{E}\longrightarrow \mathcal{Q}\longrightarrow 0
	\]
	we have \(\ch_1(\mathcal{E})=\ch_1(\mathcal{Q})\).
	By \zcref{rmk:divisor_fitting_support}, this equals \(\mathcal{Z}_f(\mathcal{Q})\).
\end{proof}

\subsection*{A duality lemma}

\begin{corollary}
	Let \(\mathcal{E},s\) be a sheaf stable pair on a smooth projective variety \(Z\).
	Then its reflexive hull together with the induced morphism
	\[
		\mathcal{O}_{Z}^r\xlongrightarrow{s^{\vee\vee}}\mathcal{E}^{\vee\vee}
	\]
	is again a sheaf stable pair.
	Moreover, the cokernel sheaf \(\mathcal{Q}'=\coker(s^{\vee\vee})\) has the same Fitting support as
	\(\mathcal{Q}=\coker(s)\).
\end{corollary}

\begin{lemma}\label{lem:duality}
	Let \(Z\) be a smooth projective variety of dimension \(n\).
	Assume that \(\mathcal{E}\) is a locally free sheaf on \(Z\),
	and \(\mathcal{T}\) is a torsion sheaf on \(Z\) with zero-dimensional support.
	Then there is a natural isomorphism:
	\[
		\Hom(\mathcal{E}^{\vee},\SExt^n(\mathcal{T},\omega_Z)) \xlongrightarrow{\sim} \Hom(\mathcal{E},\mathcal{T})^{*},
	\]
	where \(\omega_Z\) is the canonical sheaf of \(Z\).
\end{lemma}

\begin{proof}
	This follows from adjoint property and Serre duality:
	\begin{align*}
		\Hom(\mathcal{E}^{\vee},\SExt^n(\mathcal{T},\omega_Z))
		& \cong H^0(\SExt^n(\mathcal{T},\omega_Z)\otimes \mathcal{E}) \\
		& \cong \Ext^n(\mathcal{T},\mathcal{E}\otimes\omega_Z) \\
		& \cong \Hom(\mathcal{E},\mathcal{T})^{*}. \qedhere
	\end{align*}
\end{proof}

\subsection*{Relation with Quot-schemes}

We now describe the relation between the moduli space of sheaf stable pairs and the corresponding Quot-schemes.

\begin{proposition}\label{prop:relation_quot_scheme}
	Let \(Z\) be a smooth projective surface.
	Then \(M_Z(r,\ch_1,\ch_2)\) is isomorphic to a subscheme of the Quot-scheme
	\(\Quot(\mathcal{O}_{Z}^r,0,\ch_1,-\ch_2)\).
\end{proposition}

\begin{proof}
	\emph{From a sheaf stable pair to a quotient.}
	Let \(\mathcal{E},s\) be a sheaf stable pair of rank \(r\) on \(Z\).
	We have a commutative diagram with exact rows
	\[
		\begin{tikzcd}
			0 \ar[r] & \mathcal{O}_Z^r \ar[r,"s"] \ar[d,"\cong"] & \mathcal{E} \ar[r] \ar[d,hookrightarrow] & \mathcal{Q}
			\ar[r] \ar[d] & 0 \\
			0 \ar[r] & \mathcal{O}_Z^r \ar[r,"s^{\vee\vee}"] & \mathcal{E}^{\vee\vee} \ar[r] & \mathcal{Q}' \ar[r] & 0
		\end{tikzcd}
	\]
	where \(\mathcal{Q}'=\coker(s^{\vee\vee})\).
	By the snake lemma, \(\mathcal{Q}\to\mathcal{Q}'\) is injective,
	and the middle and right vertical arrows have a common cokernel \(\mathcal{T}'\).
	Applying \(\SHom(-,\mathcal{O}_{Z})\) to the top row, we obtain
	\[
		0\longrightarrow \mathcal{E}^{\vee}\xlongrightarrow{s^{\vee}} \mathcal{O}_{Z}^r\longrightarrow
		\mathcal{F} \longrightarrow 0
	\]
	where \(\mathcal{F}\) is the image of \(\mathcal{O}_{Z}^r\to\SExt^1(\mathcal{Q},\mathcal{O}_{Z})\).
	By \zcref{lem:duality} there is a surjection
	\(\mathcal{E}^{\vee}\to\mathcal{T}\coloneqq\SExt^2(\mathcal{T}',\omega_Z)\).
	Let \(\mathcal{K}\) be its kernel.
	We obtain a commutative diagram with exact rows
	\[
		\begin{tikzcd}
			0 \ar[r] & \mathcal{K} \ar[r] \ar[d] & \mathcal{O}_{Z}^r \ar[r] \ar[d,"="] & \mathcal{L}
			\ar[r] \ar[d] & 0 \\
			0 \ar[r] & \mathcal{E}^{\vee} \ar[r,"s^{\vee}"] & \mathcal{O}_{Z}^r \ar[r] & \mathcal{F} \ar[r] & 0
		\end{tikzcd}
	\]
	where \(\mathcal{L}\) is the cokernel of the composition
	\(\mathcal{K}\hookrightarrow \mathcal{E}^{\vee}\xrightarrow{s^{\vee}}\mathcal{O}_{Z}^r\).
	Then \(\mathcal{L}\to\mathcal{F}\) is surjective with kernel isomorphic to \(\mathcal{T}\).
	Thus the quotient \(\mathcal{O}_{Z}^r\to \mathcal{L}\) defines a point of
	\(\Quot(\mathcal{O}_{Z}^r,0,\ch_1(\mathcal{E}),\ch_2(\mathcal{E}))\).

	\emph{From a quotient to a sheaf stable pair.}
	Conversely, start with a quotient \(\mathcal{O}_{Z}^r\to \mathcal{L}\) and
	let \(\mathcal{K}\) be its kernel.
	Let \(\mathcal{T}\) be the maximal torsion subsheaf of \(\mathcal{L}\),
	so that \(\mathcal{F}\coloneqq\mathcal{L}/\mathcal{T}\) is pure (cf.~\cite[Definition~1.1.4]{huybrechts2010geometry}).
	We have a commutative diagram with exact rows
	\[
		\begin{tikzcd}
			0 \ar[r] & \mathcal{K} \ar[r] \ar[d] & \mathcal{O}_{Z}^r \ar[r] \ar[d,"="] & \mathcal{L}
			\ar[r] \ar[d] & 0 \\
			0 \ar[r] & \mathcal{K}' \ar[r,"t"] & \mathcal{O}_{Z}^r \ar[r] & \mathcal{F} \ar[r] & 0.
		\end{tikzcd}
	\]
	Again by the snake lemma, the cokernel of \(\mathcal{K}\to\mathcal{K}'\) is isomorphic to \(\mathcal{T}\).
	By \zcref{lem:duality}, there is an induced surjection
	\(\mathcal{K}'^{\vee}\to\mathcal{T}'\coloneqq\SExt^2(\mathcal{T},\omega_Z)\).
	If the composition \(\mathcal{O}_{Z}^r\xrightarrow{t^{\vee}} \mathcal{K}'^{\vee}\to \mathcal{T}'\) is zero,
	then letting \(\mathcal{E}\) be the kernel of \(\mathcal{K}'^{\vee}\to \mathcal{T}'\)
	gives an injection \(\mathcal{O}_{Z}^r\to \mathcal{E}\).
	Thus we obtain a commutative diagram with exact rows
	\[
		\begin{tikzcd}
			0 \ar[r] & \mathcal{O}_Z^r \ar[r,"s"] \ar[d,"="] & \mathcal{E} \ar[r] \ar[d] & \mathcal{Q}
			\ar[r] \ar[d] & 0 \\
			0 \ar[r] & \mathcal{O}_Z^r \ar[r] & \mathcal{K}'^{\vee} \ar[r] & \mathcal{Q}' \ar[r] & 0.
		\end{tikzcd}
	\]
	Then \(\mathcal{E},s\) is a sheaf stable pair and
	\([\mathcal{E},s]\in M_Z(r,\ch_1(\mathcal{L}),\ch_2(\mathcal{L}))\).

	One checks that the respective equivalence relations on sheaf stable pairs and on quotients are compatible,
	which shows that \(M_Z(r,\ch_1,\ch_2)\) embeds as a subscheme of \(\Quot(\mathcal{O}_{Z},0,\ch_1,\ch_2)\).
\end{proof}

\begin{remark}\label{rmk:locally_free_equivalence}
	In the correspondence above,
	\(\mathcal{E}\) is locally free if and only if \(\mathcal{K}\) is locally free,
	if and only if \(\mathcal{L}\) is pure.
\end{remark}

\subsection*{Associated (stable minimal) models}

\begin{definition}
	Assume \([\mathcal{E},s]\) is a stable pair of rank \(r\) on a variety \(Z\),
	with \(\mathcal{E}\) \emph{locally free}.
	Let \(s_{1}, \dots, s_{r}\) be the sections of \(\mathcal{E}\) determined by \(s\).
	We set
	\[
		\begin{array}{l}
			X = \mathbb{P}(\mathcal{E}) \xlongrightarrow{f} Z, \\
			\mathcal{O}_{X}(1) = \text{the associated invertible sheaf}, \\
			D_{i} = \text{divisor of } s_{i}, \\
			A = \text{divisor of } s_1+\dots+s_r.
		\end{array}
	\]
	Here we view \(s_{i}\) as sections of \(\mathcal{O}_{X}(1)\),
	so \(D_{i}\) is the divisor of this section (similarly for \(s_1+\dots+s_r\)).
	We call
	\[
		X, D_1,\dots,D_r, A \xlongrightarrow{f} Z
	\]
	the \emph{model associated to \([\mathcal{E},s]\)}.
\end{definition}

It is possible to modify the above model birationally and get a stable minimal model over the whole \(Z\).

\begin{theorem}\label{thm:existence_stable_minimal_model}
	Let \(\mathcal{E},s\) be a sheaf stable pair of rank \(r\) on a smooth projective surface \(Z\).
	Then the associated stable minimal model
	\[
		(X',B'), A' \longrightarrow Z'
	\]
	exists and it depends only on \(Z,r\) and \(\Supp\mathcal{Q}\).
\end{theorem}

\begin{proof}
	Let \(Q=\Supp\coker(s)\) and \(U=Z\setminus Q\).
	Over \(U\), the restricted sheaf stable pair \(\mathcal{E}|_U,s|_U\) has an associated model
	\[
		X^U,D^U_{1}, \dots, D^U_{r}, A^U\longrightarrow U.
	\]

	Let \(\pi\colon (\widetilde{Z},\widetilde{Q})\to (Z,Q)\) be the minimal log resolution
	(\cite[Proposition~1.13.7]{kawamata2024algebraic}).
	Then \(\widetilde{Z}\) is smooth and \(\widetilde{Q}\) is simple normal crossing.
	The pullback	\(\pi^{*}\mathcal{E},\pi^{*}s\) is a sheaf stable pair on \(\widetilde{Z}\),
	with \(\Supp\coker(\pi^{*}s)=\widetilde{Q}\).
	Consider the identity morphism \(\mathcal{O}_{\widetilde{Z}}^r\to\mathcal{O}_{\widetilde{Z}}^r\)
	and its associated model
	\[
		\widetilde{X}',\widetilde{D}'_{1},\dots,\widetilde{D}'_{r},\widetilde{A}'\xlongrightarrow{\widetilde{f}}
		\widetilde{Z}.
	\]
	Let \(\widetilde{B}'=\sum\widetilde{D}_i'+\widetilde{f}'^{*}\widetilde{Q}\).
	Then \(\widetilde{X}'\) is a compactification of \(X^U\) over \(Z\),
	and the pair \((\widetilde{X}',\widetilde{B}'+\widetilde{A}')\) is log smooth.
	Since
	\[
		K_{\widetilde{X}'}+\widetilde{B}'=\widetilde{f}^{*}(K_{\widetilde{Z}}+\widetilde{Q}),
	\]
	any \(K_{\widetilde{X}'}+\widetilde{B}'\)-MMP over \(Z\) is induced by
	an \(K_{\widetilde{Z}}+\widetilde{Q}\)-MMP on \(\widetilde{Z}\) over \(Z\),
	while the latter is also an \(K_{\widetilde{Z}}+\widetilde{Q}\)-MMP over \(Z\).
	But \(K_{\widetilde{Z}}+\widetilde{Q}\) is already nef and semi-ample over \(Z\);
	hence any such MMP is trivial.
	Thus \(K_{\widetilde{X}'}+\widetilde{B}'\) is semi-ample over \(Z\),
	defining a contraction \(\widetilde{X}'\to Z'\to Z\),
	where \((Z',Q')\) is the relative lc model of \((Z,Q)\) (cf.\cite[Theorem~7.2]{fujino2012minimal}).
	Moreover, \(\widetilde{A}'\) is semi-ample over \(Z'\), inducing a birational contraction
	\(\widetilde{X}'=\widetilde{Z}\times\mathbb{P}^1\to Z'\times\mathbb{P}^1\).

	Now consider the identity morphism \(\mathcal{O}_{Z'}^r\to \mathcal{O}_{Z'}^r\) and its associated model
	\[
		X',D'_{1}, \dots, D'_{r},A' \xlongrightarrow{f'} Z'.
	\]
	Let \(B'=\sum D_i'+f'^{*}Q'\).
	Then \((X',B'),A'\xrightarrow{f'}Z'\) is the associated stable minimal model of \(\mathcal{E},s\),
	and by construction it depends only on \(Z\), \(r\) and \(Q=\Supp\coker(s)\).
\end{proof}

Finally, viewing each \(s_i\) as a section of \(\mathcal{O}_X(1)\),
the common zero locus
\[
	T\coloneqq V(s_1,\dots,s_r)=D_1\cap\dots\cap D_r
\]
maps onto \(\mathcal{Z}_f(\mathcal{Q})\) via \(f\)
(see \cite[\S~14.4]{fulton1998intersection}).

\begin{proposition}\label{prop:image_of_intersection}
	In the case \(r=2\), assume \(D_1\) and \(D_2\) have no common (vertical) component.
	Then \(f_*(D_1\cdot D_2)=\mathcal{Z}_f(\mathcal{Q})\) as cycles on \(Z\).
\end{proposition}

Since \(D_1,D_2\) have no common component,
the intersection \(T=D_1\cap D_2\) is of codimension \(2\) in \(X\).
This is clear for the horizontal component of \(T\) since it is mapped to \(Q\).

The localised Chern class of \(f^{*}\mathcal{E}\otimes\mathcal{O}_{X}(-1)\) is \([V(s)]\),
which is also horizontal.
By \cite[Theorem~14.4.(c)]{fulton1998intersection},
we obtain the equality of cycles \(f_*([T])=[Q]\).
Indeed, we can find some divisor in the linear system \(\mathcal{O}_{X}(1)\),
which is purely horizontal (see \zcref{lem:pure_horizontal} below).
This means \(f_*\) maps \([T]\) isomorphically to \([Q]\).

\section{Geometric treatment}\label{sec:geometric_treatment}

Let \(Z\) be a normal algebraic variety.
Given a sheaf stable pair \(\mathcal{E},s\) on \(Z\) with \(\mathcal{E}\) locally free,
we have the following rigidity statement.

\begin{lemma}
	Suppose that \(\mathcal{E},s\) and \(\mathcal{G},t\) are sheaf stable pairs on a normal variety \(Z\),
	with \(\mathcal{E}\) and \(\mathcal{G}\) locally free.
	Let
	\[
		X,D_{1}, \dots, D_{r},A \to Z \quad \text{and} \quad Y,E_{1}, \dots, E_{r},C \to Z
	\]
	be the associated models of \(\mathcal{E},s\) and \(\mathcal{G},t\), respectively.
	If \(X,D_{1}, \dots, D_{r},A\) and \(Y,E_{1}, \dots, E_{r},C\) are isomorphic over a big open subset \(U\subseteq Z\),
	then \([\mathcal{E},s]=[\mathcal{G},t]\).
\end{lemma}

\begin{proof}
	By assumption, the restricted sheaf stable pairs \(\mathcal{E}|_U,s|_U\) and \(\mathcal{G}|_U,t|_U\) are isomorphic.
	Since both \(\mathcal{E}\) and \(\mathcal{G}\) are locally free, hence reflexive, and \(Z\) is normal,
	this isomorphism extends uniquely to all of \(Z\).
\end{proof}

Let \(Z\) be a normal algebraic variety.
Let \(\mathcal{E},s\) be a sheaf stable pair of rank \(2\) on \(Z\) with \(\mathcal{E}\) locally free,
and let
\begin{equation}\label{eq:associated_model}
	X,D_1,D_2,A \xlongrightarrow{f} Z
\end{equation}
be the associated model.
Assume that the cokernel scheme \(Q\) is connected.
In particular, the singular set of \(Q_{\red}\) is of codimension two in \(Z\).

\begin{remark}\label{rmk:vertical_image}
	Assume that \(D_1=\div(s_1)\) and \(D_2=\div(s_2)\) have no common (vertical) components.
	Then for any \(s\in\Span\{s_1,s_2\}\), if \(\div(s)\) has a vertical component \(V\),
	the image of \(V\) under \(f\) is contained in \(Q\).
	Otherwise, \(V\) will meet the horizontal part of \(D_1\), contradicting \zcref{prop:image_of_intersection}.
\end{remark}

\begin{lemma}\label{lem:pure_horizontal}
	Assume \(D_1=\div(s_1)\) and \(D_2=\div(s_2)\) have no common (vertical) components.
	Then we can find a section
	\(s_h\in\Span\{s_1,s_2\}\) such that \(A_h=\div(s_h)\) is purely horizontal.
\end{lemma}

\begin{proof}
	If either \(\div(s_1)\) or \(\div(s_2)\) is purely horizontal, we are done.
	Assume both \(D_1\) and \(D_2\) have vertical components.
	Consider first the section \(t_1=s_1+s_2\).
	If \(T_1\coloneqq\div(t_1)\) has no vertical components, set \(s_h=t_1\).
	Otherwise \(T_1\) has vertical components.
	Since \(D_1\) and \(D_2\) have no common vertical components,
	neither does \(T_1\) with \(D_i\) for \(i=1,2\).

	Let \(Q_{11},\dots,Q_{1m_1}\) be those components of \(Q\) over which no vertical components
	of \(D_1\) or \(D_2\) maps via \(f\).
	By \zcref{rmk:vertical_image}, every vertical component of \(T_1\) maps into \(\bigcup_{j=1}^{m_1}Q_{1j}\).
	Next, let \(Q_{21},\dots,Q_{2m_2}\) be the components of \(Q\) over which no vertical components
	of \(D_1\), \(D_2\) or \(T_1\) maps to.
	Note that \(\{Q_{21},\dots,Q_{2m_2}\}\) is a proper subset of \(\{Q_{11},\dots,Q_{1m_1}\}\).
	Hence \(m_2<m_1\).

	Now take \(t_2=s_1+2s_2\) and set \(T_2\coloneqq\div(t_2)\).
	If \(T_2\) has no vertical components, set \(s_h=t_2\).
	Otherwise, as above, \(T_2\) has no common vertical components with \(D_1\), \(D_2\), or \(T_1\).
	By \zcref{rmk:vertical_image}, the vertical components of \(T_2\) map into \(\bigcup_{j=1}^{m_2}Q_{2j}\).

	Define \(Q_{31},\dots,Q_{3m_3}\) to be the components of \(Q\) over which no vertical component
	of \(D_1\), \(D_2\), \(T_1\), \(T_2\) maps.
	Then \(m_3<m_2\).

	Iterating this with \(t_j=s_1+js_2\),
	we obtain a strictly decreasing sequence of integers \(m_1>m_2>\dots\geq 0\).

	Either at some stage \(T_j\) has no vertical components,
	in which case we set \(s_h=t_j\), or the process terminates when \(m_k=0\).
	In the latter case, every component of \(Q\) is the image of a vertical component of one of
	\(D_1\), \(D_2\), \(T_1,\dots,T_{k-1}\).
	Then \(t_k=s_1+ks_2\) has no vertical components by \zcref{rmk:vertical_image}, so we set \(s_h=t_k\).
\end{proof}

In the following we assume that \(Z\) is a normal projective surface.
Take the minimal log resolution \(\pi\colon\widetilde{Z}\to Z\) of \((Z,Q)\)
\cite[Proposition~1.13.7]{kawamata2024algebraic}.
Write
\[
	\widetilde{Q}\coloneqq\pi^{*}Q=Q^{\sim}+C
\]
where \(Q^{\sim}\) is the strict transform of \(Q\) and \(C\) is \(\pi\)-exceptional.
Let
\[
	\widetilde{X},\widetilde{D}_1, \widetilde{D}_2,\widetilde{A} \xlongrightarrow{f} \widetilde{Z}
\]
be the base change of \eqref{eq:associated_model};
this is the associated model of the sheaf stable pair \((\pi^{*}\mathcal{E},\pi^{*}s)\) on \(\widetilde{Z}\).
We denote the induced morphism \(\widetilde{X}\to X\) by \(\pi_{X}\),
and write \(\widetilde{\mathcal{E}}\coloneqq\pi^{*}\mathcal{E}\).

In what follows, we always assume that \(D_1\) and \(D_2\) have no common (vertical) component
(this holds, for instance, if \(Q\) is reduced).
By \zcref{lem:pure_horizontal}, there exists some section \(s_h\in\Span\{s_1,s_2\}\)
such that \(A_h=\div(s_h)\) is purely horizontal.
Then \(f|_{A_h}\colon A_h\to Z\) is an isomorphism, and thus
\[
	T\coloneqq D_1\cap D_2=A_h\cap D_i
\]
is mapped isomorphically to \(Q\) for \(i=1,2\).
Moreover, \(\widetilde{A}_h\) contains no \(\pi\)-exceptional divisor,
so \(\widetilde{A}_h\) is purely horizontal as well, and
\[
	\widetilde{T}=\widetilde{D}_1\cap\widetilde{D}_2=\widetilde{A}_h\cap\widetilde{D}_i
\]
is mapped isomorphically to \(\widetilde{Q}\) for \(i=1,2\).

\subsection*{Forward transform via elementary transformations}

We now perform a sequence of Maruyama's elementary transformations \cite[Theorem~1.4]{maruyama1982elementary},
to construct intermediate models
\[
	X_i, D_{1,i}, D_{2,i}, A_i \xlongrightarrow{f_i} \widetilde{Z},
\]
eventually arriving at a stable minimal model
\[
	X', D_1', D_2', A' \xlongrightarrow{f'} Z'.
\]

For clarity, we first illustrate this when the Fitting support of \(\widetilde{\mathcal{Q}}\) is \(Q=H_1+H_2\).
Here either \(Q\) has two distinct irreducible components \(H_1\) and \(H_2\),
or \(Q=2H_1\) with \(H_1=H_2\).
In this situation we have the commutative diagram
\begin{equation}\label{eq:maruyama_transform_h1_h2}
	\begin{tikzcd}
		& & 0 \ar[d] & 0 \ar[d] \\
		0 \ar[r] & \mathcal{O}_{\widetilde{Z}}^2 \ar[r,"s_1"] \ar[d,"\cong"] & \widetilde{\mathcal{E}}'
		\ar[r,"\delta_1"] \ar[d] & \widetilde{\mathcal{Q}}_{H_2}(-H_1)
		\ar[r] \ar[d] & 0 \\
		0 \ar[r] & \mathcal{O}_{\widetilde{Z}}^2 \ar[r] & \widetilde{\mathcal{E}} \ar[r] \ar[d,"\delta"] &
		\widetilde{\mathcal{Q}} \ar[r] \ar[d] & 0 \\
		& & \widetilde{\mathcal{Q}}_{H_1}\ar[r,"\cong"] \ar[d] & \widetilde{\mathcal{Q}}_{H_1} \ar[d] \\
		& & 0 & 0
	\end{tikzcd},
\end{equation}
where \(\widetilde{\mathcal{E}}'=\ker(\delta)\) and the right column is obtained by tensoring
\(\widetilde{\mathcal{Q}}\) with the short exact sequence
\[
	0\to \mathcal{O}_{H_2}(-H_1)\to \mathcal{O} \to \mathcal{O}_{H_1}\to 0
\]
for two prime divisors \(H_1\) and \(H_2\) on the smooth projective variety \(Z\).
Now we consider the middle column and run the Maruyama's elementary transformation for
\(X_0\coloneqq\mathbb{P}(\widetilde{\mathcal{E}})\) along \(Y_0\coloneqq\mathbb{P}(\widetilde{\mathcal{Q}}_{H_1})\)
to get \(X_1\coloneqq\mathbb{P}(\widetilde{\mathcal{E}}')\) (\cite[Theorem~1.4]{maruyama1982elementary}).
That is to say, we blow up \(X_0\) along \(Y_0\) and then contract the
proper transformation of \((X_0)_{H_1}\),
which is the restriction of the projective bundle \(\mathbb{P}(\widetilde{\mathcal{E}})\) to \(H_1\);
see the following diagram:
\[
	\begin{tikzcd}
		{} & \Bl_{Y_0}X_0 \ar[ld,start anchor=south west] \ar[rd,start anchor=south east] \\
		X_0=\mathbb{P}(\widetilde{\mathcal{E}}) \ar[d,"f"'] & & X_{1}\simeq\mathbb{P}(\widetilde{\mathcal{E}}')
		\ar[d,"f_{1}"] \\
		\widetilde{Z} & & \widetilde{Z}.
	\end{tikzcd}
\]
Let \(D_{1,1}\) and \(D_{2,1}\) be the strict transforms of \(\widetilde{D}_1\) and \(\widetilde{D}_2\).
The pair \(\widetilde{\mathcal{E}}',s_1\) is a sheaf stable pair with cokernel \(\widetilde{\mathcal{Q}}_{H_2}(-H_1)\),
which we denote by \(\mathcal{Q}_1\).
Then by \zcref{prop:image_of_intersection},
we know \(f_1\) maps \(T_1\coloneqq D_{1,1}\cap D_{2,1}\) isomorphically onto the \(Q_1\coloneqq H_2\)
which is the fitting support of \(\mathcal{Q}_1\).
Now \(Q_1=H_2\leq Q=H_1+H_2\) and
we see that the multiplicity of \(H_1\) decreases by one after the Maruyama's elementary transformation.

Then we consider the top row in commutative diagram \eqref{eq:maruyama_transform_h1_h2} and run the
Maruyama's elementary transformation for \(X_1\) along \(Y_1\coloneqq\mathbb{P}(\mathcal{Q}_1)\) to get
\[
	X_2\coloneqq\mathbb{P}(\ker(\delta_1))=\mathbb{P}(\mathcal{O}_{\widetilde{Z}}^2)\simeq
	\widetilde{Z}\times\mathbb{P}^1.
\]
Let \(D'_1\) and \(D'_2\) be the strict transforms of \(D_{1,1}\) and \(D_{2,1}\).
Now we consider the sheaf stable pair \(\mathcal{O}_{\widetilde{Z}}^2,\id\) corresponding to
the identity morphism \(\mathcal{O}_{\widetilde{Z}}^2\to\mathcal{O}_{\widetilde{Z}}^2\) with cokernel zero.
Hence \(T_2\coloneqq D'_{1}\cap D'_{2}=\emptyset\).
This is the situation referred to as Step 3 in the forward construction below.

\medskip

We now describe the general forward procedure.

\begin{equation}\label{eq:geometric_interpretation}
	\begin{tikzcd}
		\widetilde{X} \ar[r,dashed] \ar[d,"\widetilde{f}"']\ar[rd,"\pi_X"] & \cdots \ar[r,dashed]
		& \cdots \ar[r,dashed] & \widetilde{X}'\ar[d,"\widetilde{f}'"]\ar[ld,"\theta_{X'}"'] \\
		\widetilde{Z} \ar[rd,"\pi"'] & X \ar[d,"f"'] & X'\ar[d,"f'"]\ar[l] & \widetilde{Z}' \ar[ld,"\theta"] \\
		& Z & \ar[l,"\pi'"] Z'
	\end{tikzcd}
\end{equation}

\noindent
\textbf{Forward Transform.}
Let \(X_0,D_{1,0},D_{2,0},A_0\) denote the model \(\widetilde{X},\widetilde{D}_1,\widetilde{D}_2,\widetilde{A}\).
Write
\[
	Q_0=\widetilde{Q}=\sum_{j=1}^{l}n_{j,0}H_j,
\]
where \(H_j\) are distinct irreducible components of \(Q_0\).

\medskip

Steps 0-2 below describe the process expressed by the dashed arrows from \(\widetilde{X}\)
to \(\widetilde{X}'\) over \(\widetilde{Z}\) in the diagram \eqref{eq:geometric_interpretation}.
Step 3 describes the contractions \(\theta_{X'}\) and \(\theta\) in the diagram \eqref{eq:geometric_interpretation}.

\medskip

\noindent
\textit{Step 0 (Initial Data).}
Assume we have constructed
\[
	X_i, D_{1,i}, D_{2,i}, A_i \xlongrightarrow{f_i} \widetilde{Z}
\]
with the following properties:
\begin{itemize}[leftmargin=*]
	\item \(f_i\colon X_i\to \widetilde{Z}\) is a \(\mathbb{P}^1\)-bundle;
	\item \(T_i\coloneqq D_{1,i}\cap D_{2,i}\) maps isomorphically onto a subscheme
		\(Q_i\subseteq\widetilde{Q}\) under \(f_i\);
	\item writing \(Q_i=\sum_{j=1}^{l}n_{j,i}H_j\), the total multiplicity \(N_i=\sum_{j=1}^{l}n_{j,i}\geq 0\).
\end{itemize}

\medskip

\noindent
\textit{Step 1 (Elementary Transformation).}
If \(T_i=\emptyset\), then we proceed directly to Step 3.

If \(T_i\neq\emptyset\), pick a component \(S\subseteq Q_i\).
Consider the restriction \((X_i)_S\to S\) of the \(\mathbb{P}^1\)-bundle \(f_i\) to \(S\).
By construction and \zcref{prop:image_of_intersection},
\((X_i)_S\) contains a component \(S'\) of \(T_i\) with multiplicity \(\mu_{Q_i}S\).
We perform Maruyama's elementary transformation:
blow up \(X_i\) along \(S'\), then contract the strict transform of \((X_i)_S\).
The resulting contraction is the blowup of \(X_{i+1}\) along a curve \(S''\).
By \cite[Theorem~1.4]{maruyama1982elementary},
the new morphism \(f_{i+1}\colon X_{i+1}\to \widetilde{Z}\) is again a \(\mathbb{P}^1\)-bundle,
fitting into
\[
	\begin{tikzcd}
		{} & \Bl_{S'}X_i = \Bl_{S''}X_{i+1} \ar[ld,start anchor=south west] \ar[rd,start anchor=south east] \\
		X_i \ar[d,"f_i"'] & & X_{i+1} \ar[d,"f_{i+1}"] \\
		\widetilde{Z} \ar[rr,equal] & & \widetilde{Z}
	\end{tikzcd}
\]

\noindent
\textit{Step 2 (Subsequent Transformations).}
Let \(D_{1,i+1}\), \(D_{2,i+1}\) and \(A_{i+1}\) be the strict transforms of \(D_{1,i}\), \(D_{2,i}\) and \(A_i\).
Set \(T_{i+1}\coloneqq D_{1,i+1}\cap D_{2,i+1}\).
Then the total multiplicity decreased by one: \(N_{i+1}=N_i-1\).

If \(S\subseteq Q_i\) is a reduced component, i.e., \(S\) has multiplicity one in \(Q_i\),
then \(T_{i+1}\) has no component contained in \((X_{i+1})_S\).
In this case, return to Step~0 with
\[
	X_{i+1}, D_{1,i+1}, D_{2,i+1}, A_{i+1} \xlongrightarrow{f_{i+1}} \widetilde{Z}.
\]

If \(S\subseteq Q_i\) has multiplicity \(>1\),
then \(T_{i+1}\) contains a component \(S''\subseteq (X_{i+1})_S\).
In this case, return to Step~1, again with the same component \(S\).

\noindent
\textit{Step 3 (Final Result).}
The process terminates after finitely many steps, since \(N_i\) strictly decreases.
Set
\[
	\begin{cases}
		(\widetilde{X}',\widetilde{D}_1',\widetilde{D}_2',\widetilde{A}') = (X_i,D_{1,i},D_{2,i},A_i), \\
		\widetilde{f}'=f_i, \\
		\widetilde{B}'=\widetilde{D}_1'+\widetilde{D}_2'+\widetilde{f}'^{*}\Supp \widetilde{Q}.
	\end{cases}
\]
Since \(\widetilde{D}_1'\sim \widetilde{D}_2'\) and \(\widetilde{D}_1'\cap \widetilde{D}_2'=\emptyset\),
the linear system \(\abs{\widetilde{D}_1'}=\abs{\widetilde{D}_2'}\) defines an isomorphism
\(\widetilde{X}'\simeq\widetilde{Z}\times\mathbb{P}^1\).

We then apply the construction of \zcref{thm:existence_stable_minimal_model} to
produce a stable minimal model \((X',B'), A' \longrightarrow Z'\).
Here \(Z'\) is obtained as the relative lc model of \((\widetilde{Z},\Supp\widetilde{Q})\) over \(Z,\Supp Q\);
and \(B'=D_1'+D'_2+f'^{*}\Supp Q'\) is constructed from the model \(X',D'_{1}, D'_{2},A'\xlongrightarrow{f'} Z'\)
associated to the identity morphism \(\mathcal{O}_{Z'}^2\to\mathcal{O}_{Z'}^2\).

\medskip

To recover \(X,D_1,D_2,A\to Z\) from the associated stable minimal model \(X',B',A'\to Z'\),
we perform a sequence of Maruyama's elementary transformations in the reverse direction.

\medskip
\noindent
\textbf{Backward Transform.}
Take the minimal log resolution \(\theta\colon\widetilde{Z}\to Z'\) of \((Z',Q')\)
(\cite[Proposition~1.13.7]{kawamata2024algebraic}).
The induced morphism \(\pi\colon\widetilde{Z}\to Z\) is also the minimal log resolution of \((Z,Q)\).
Write
\[
	\widetilde{Q}\coloneqq\theta^{*}Q'=Q'^{\sim}+C'
\]
where \(Q'^{\sim}\) is the strict transform of \(Q'\) and \(C'\) is \(\theta\)-exceptional.
Let
\[
	\widetilde{X}',\widetilde{D}'_1,\widetilde{D}'_2,\widetilde{A}' \xlongrightarrow{\widetilde{f}'} \widetilde{Z}
\]
be the base change of \(X',D_1',D_2',A'\to Z'\),
which is the associated model of the sheaf stable pair \((\mathcal{O}_{\widetilde{Z}}^2,\id)\) on \(\widetilde{Z}\).
Denote the induced map \(\widetilde{X}'\to X'\) by \(\theta_{X'}\).

Let \(X^0,D^0_1,D^0_2,A^0\) represent the model \(\widetilde{X}',\widetilde{D}'_1,\widetilde{D}'_2,\widetilde{A}' \).
Pick a fibre \(X^0_p\) of \(X^0\to\mathbb{P}^1\) over some point \(p\in\mathbb{P}^1\);
then \(X^0_p\simeq \widetilde{Z}\).

Steps 0-2 below describe the process from \(\widetilde{X}'\) to \(\widetilde{X}\) over \(\widetilde{Z}\),
inverse to the sequence of dashed arrows in \eqref{eq:geometric_interpretation}.
Step 3 accounts for the contractions \(\pi_{X}\) and \(\pi\).

\medskip

\noindent
\textit{Step 0 (Initial Data).}
Assume we have constructed
\[
	X^i, D^i_1, D^i_2, A^i \xlongrightarrow{f^i} \widetilde{Z}
\]
with:
\begin{itemize}[leftmargin=*]
	\item \(f^i\colon X^i\to \widetilde{Z}\) is a \(\mathbb{P}^1\)-bundle;
	\item \(X^i_p\subseteq X^i\) is the proper transform of \(X^0_p\);
	\item \(T^i\coloneqq D^i_1\cap D^i_2\) maps isomorphically to a subscheme \(Q^{i}\subseteq\widetilde{Q}\)
		under \(f_i\).
\end{itemize}

\noindent
\textit{Step 1 (Elementary Transformation).}
If \(Q^i=\widetilde{Q}\), proceed directly to Step 3.

If \(Q^i\neq\widetilde{Q}\), choose a component \(S\subseteq\widetilde{Q}-Q^i\) (as divisors).
By construction, \(X^i_S\) meets \(X_p^i\) at some curve \(S'\),
where \(X^i_S\to S\) is the restriction of the \(\mathbb{P}^1\)-bundle \(f_i\) to \(S\).
Perform Maruyama's elementary transformation:
blowup \(X^i\) along \(S'\) and then contract the strict transform of \(X^i_S\).
The resulting morphism \(f^{i+1}\colon X^{i+1}\to\widetilde{Z}\) is again a \(\mathbb{P}^1\)-bundle,
fitting into
\[
	\begin{tikzcd}
		{} & \Bl_{S'}X^i = \Bl_{S''}X^{i+1} \ar[ld,start anchor=south west] \ar[rd,start anchor=south east] \\
		X^i \ar[d,"f^i"'] & & X^{i+1} \ar[d,"f^{i+1}"] \\
		\widetilde{Z} \ar[rr,equal] & & \widetilde{Z}
	\end{tikzcd}
\]

\noindent
\textit{Step 2 (Subsequent Transformations).}
Let \(D_1^{i+1}\), \(D_2^{i+1}\) and \(A^{i+1}\) be the strict transforms of \(D_1^i\), \(D_2^i\) and \(A^i\),
and set \(T^{i+1}\coloneqq D_1^{i+1}\cap D_2^{i+1}\).

Suppose \(S\subseteq\widetilde{Q}-Q^i\) has multiplicity \(\geq 1\).
Then \(T^{i+1}\) has a component \(S''\subseteq X^{i+1}_S\) with multiplicity \(\leq\mu_{\widetilde{Q}}S-1\).
Here $S''$ is the strict transform of $S'$ after Step~1.
Next we perform Maruyama's elementary transformation along a section of \(X^{i+1}_S\to S\)
disjoint with \(S''\) but passing through the points of
\[
	X^i_p\cap (f^{i+1})^{-1}((\widetilde{Q}_{\red}-(Q^i+S)_{\red})\cap S).
\]
By \zcref{lem:parameter_space_cross_sections}, such sections form a family
(typically parameterised by an affine space).

\begin{remark}\label{rmk:parametrize_back}
	Assume that \(S\) has self-intersection \(b\) in \(\widetilde{Z}'\) and \(S'\) has self-intersection \(a\)
	in \(X_S^i\).
	After blowing up \(X^i_S\) along \(S'\),
	let \(E\) be the exceptional divisor and \(C\) the intersection of \(E\) with the strict transform of \(X^i_S\).
	By \zcref{prop:self_intersection_blowup_curve}, we have \((C)^2_E=b-a\).
	Hence \(E\) is a ruled surface of with invariant \(|a-b|\), by \zcref{cor:ruled_surface_inv}.
	The curve \(S''\) lies on \(E\) and is disjoint with \(C\).

	Suppose further that \(S\) is a smooth rational curves.
	If \(b-a>0\), then \(S''\) is the unique negative section of \(E\to S\);
	if \(b-a<0\), the possible choices of \(S''\) is parametrized by \(\mathbb{A}^{\abs{a-b}}\)
	(\zcref{lem:parameter_space_cross_sections}).
\end{remark}

After finite many such steps, \(S\) ceases to be a component of \(\widetilde{Q}-Q^i\),
and \(T^{i+1}\) acquires a component on \(X^{i+1}_S\) with multiplicity \(\mu_{\widetilde{Q}} S\).
Then we return to Step~0 with
\[
	X^{i+1}, D^{i+1}_1, D^{i+1}_2, A^{i+1} \xlongrightarrow{f^{i+1}} \widetilde{Z}
\]
and continue.

\noindent
\textit{Step 3 (Final Outcome).}
The procedure terminates after finitely many steps.
Set
\[
	\begin{cases}
		(\widetilde{X},\widetilde{D}_1,\widetilde{D}_2,\widetilde{A}) = (X^i,D_1^i,D_2^i,A^i), \\
		\widetilde{f}=f^i.
	\end{cases}
\]
Then \(\widetilde{X},\widetilde{D}_1,\widetilde{D}_2,\widetilde{A}\xrightarrow{\widetilde{f}}\widetilde{Z}\)
is the model associated to some sheaf stable pair \(\widetilde{\mathcal{E}},\widetilde{s}\) on \(\widetilde{Z}\).
We know \(\widetilde{X}_{C}\simeq C\times\mathbb{P}^1\)
for each \(\pi\)-exceptional curve \(C\) by \zcref{lem:p1_bundle_over_exc_curve}.
Since \(\pi\colon\widetilde{Z}\to Z\) is obtained by a sequence of blowups,
one may contract \(\widetilde{X}\) and \(\widetilde{Z}\) in a controlled way to get \(X\) and \(Z\) respectively.
Let \(D_1,D_2\) and \(A\) be the strict transforms of \(\widetilde{D}_1,\widetilde{D}_2\) and \(A\) respectively.

\medskip

The class \([\mathcal{E},s]\) is uniquely determined by the associated model \(X,D_1,D_2,A\)
which is in turn determined by the fixed model \(X',D_1',D_2',A'\to \widetilde{Z}\),
a fibre of \(X'\to \mathbb{P}^1\) and choices of cross-sections as in Step~2 of the backward transform.

\begin{lemma}\label{lem:p1_bundle_over_exc_curve}
	For each \(\pi\)-exceptional curve \(C\subseteq\widetilde{Q}\), we have \(\widetilde{X}_C\simeq C\times\mathbb{P}^1\).
\end{lemma}

\begin{proof}
	Let $m_C=\mu_{\widetilde{Q}}C$.
	In the backward transform \(\widetilde{X}'\rightsquigarrow\widetilde{X}\) over \(\widetilde{Z}\),
	which is inverse to the dashed sequence in \eqref{eq:geometric_interpretation},
	we must perform the elementary transformation affecting \(C\) exactly $m_C$ times so that
	the intersection multiplicity of \(D^i_1\) and \(D^i_2\) on \(X^i_C\) becomes \(m_C\).
	From that point on, by construction, the strict transforms of \(X^i_C\) are all isomorphic
	unless we perform a Maruyama's elementary transformation along a section dominating a component that meets \(C\).

	Thus we may assume that the backward transform starts with the component \(C\).
	Let \(-b\) be the self-intersection of \(C\) in \(\widetilde{Z}\), so that \(b>0\).
	Let \(S^0\subseteq X^0_C\) be the intersection of \(X^0_C\) with \(X^0_p\),
	and let \(S^j\) denote the strict transform of \(S^0\) in \(X^j_C\) after the \(j\)-th elementary
	transformation involving \(C\).
	Write
	\[
		a_j\coloneqq (S^j)^2_{X^j_C}
	\]
	for the self-intersection of $S^j$ in $X^j_C$.
	Initially, $a_0=0$.
	By \zcref{rmk:parametrize_back}, each such transformations decreases this number by \(b\), and hence \(a_j=-bj\).
	In particular, after \(m_C\) such transformations,
	the surface \(X^{m_C}_C\) is isomorphic to the Hirzebruch surface \(\mathbb{F}_{bm_C}\).

	We now consider Maruyama's elementary transformations along sections dominating components that meet \(C\).
	We have
	\[
		\widetilde{Q}=\theta^{*}Q'=Q'^{\sim}+m_CC+\sum_{C'\neq C}m_{C'}C',
	\]
	and therefore
	\[
		bm_C=-C\cdot(m_C C) = C\cdot\left(Q'^{\sim}+\sum_{C'\neq C}m_{C'}C'\right).
	\]
	For each component \(C'\) meeting \(C\), they intersect transversally at a single point.
	Each elementary transformation along a section dominating such a component \(C'\)
	decreases the invariant of the Hirzebruch surface \(X^i_C\) by one.
	After performing all required transformations along components meeting \(C\),
	the invariant of \(X^i_C\) decreases from \(bm_C\) to \(0\).
	Consequently, the resulting surface \(X^i_C\) is isomorphic to \(C\times\mathbb{P}^1\).
\end{proof}

\section{Smooth projective surfaces of Picard number \(1\)}\label{sec:pic_one_surfaces}

In this section, we fix a smooth projective surface \(Z\) of Picard number \(1\)
with an ample generator \(H\) of \(\NS(Z)\).
We consider the Hilbert-Chow morphism \(\sigma\colon M_Z(\ch)\to \CDiv(Z)\)
given by \([\mathcal{E},s]\mapsto\mathcal{Z}_f(\mathcal{Q})\),
where \(\mathcal{Q}\coloneqq \coker(s)\) is the cokernel sheaf.

\begin{remark}
	In general, if \(\mathcal{Q}\) has rank \(1\) along its support,
	then \(\sigma\) sends \([\mathcal{E},s]\) to the (reduced) support of \(\mathcal{Q}\).
	If \(\mathcal{Q}\) has higher rank along some component,
	then \(\sigma\) maps to a non-reduced curve whose multiplicity on that component equals the rank.
\end{remark}

\begin{lemma}
	Let \(C\subseteq Z\) be a curve of degree \(dH^2\), i.e., \(C\cdot H=dH^2\).
	The Chern character of \(\mathcal{O}_{C}\) is \((0,dH,-d^2H^2/2)\).
\end{lemma}

\begin{lemma}\label{lem:torsion_free_sheaf_pullback}
	Let \(\mathcal{F}\) be a torsion-free sheaf of rank \(1\) on a curve \(C\), of degree zero,
	and suppose there is a surjection \(\mathcal{O}_{C}^r\to\mathcal{F}\to 0\) for some \(r\geq 1\),
	Then \(\mathcal{F}\simeq\mathcal{O}_{C}\).
\end{lemma}

\begin{proof}
	Let \(\pi\colon C'\to C\) be the normalisation.
	Pulling back, we obtain a surjection \(\mathcal{O}_{C'}^r\to \pi^{*}\mathcal{F}\to 0\) with
	\(\pi^{*}\mathcal{F}\) being torsion-free of degree zero.
	Then \(\pi^{*}\mathcal{F}\simeq \mathcal{O}_{C'}\).
	By assumption, there is a non-trivial morphism \(\mathcal{O}_{C}\to \mathcal{F}\)
	whose cokernel is a torsion sheaf \(\mathcal{G}\).
	Thus we have a short exact sequence
	\[
		0\longrightarrow \mathcal{O}_{C}\longrightarrow \mathcal{F}\longrightarrow \mathcal{G}\longrightarrow 0.
	\]
	Pulling back this sequence to \(C'\) and comparing the degrees,
	we find that \(\pi^{*}\mathcal{G}=0\) which forces \(\mathcal{G}=0\).
	We conclude that \(\mathcal{F}\simeq\mathcal{O}_{C}\).
\end{proof}

\begin{lemma}\label{lem:main_ineq}
	Let \(r_i\geq 1\) be integers for \(1\leq i\leq m\).
	Then
	\[
		\Big(\sum_{i=1}^m r_i\Big)^2 \geq \sum_{i=1}^m (2m+1-2i)r_i
	\]
	with equality if and only if \(r_i=1\) for all \(i\).
\end{lemma}

\begin{proof}
	Set \(S=\sum_{i=1}^m r_i\).
	Then
	\begin{align}
		\sum_{i=1}^m (2m+1-2i)r_i & = (2m+1)S - 2\sum_{i=1}^m ir_i \nonumber \\
		& \leq (2m+1)S - 2\sum_{i=1}^m i \label{eq:simple_ineq} \\
		& = (2m+1)S - m(m+1). \nonumber
	\end{align}
	Thus \(S^2\geq (2m+1)S-m(m+1)\), with equality if and only if \(S=m\) or \(S=m+1\).
	Combining this with \eqref{eq:simple_ineq}, we see that equality in the original inequality holds
	if and only if \(r_i=1\) for all \(i\).
\end{proof}

\begin{proposition}\label{prop:moduli_quot_isomorphism}
	The moduli space \(M_Z(r,dH,d^2H^2/2)\) is isomorphic to the Quot-scheme \(\Quot(\mathcal{O}_{Z}^r;0,dH,-d^2H^2/2)\).
\end{proposition}

\begin{proof}
	It suffices to show that every quotient of \(\mathcal{O}_{Z}^r\) with Chern character \((0,dH,-d^2H^2/2)\) is
	isomorphic to \(\mathcal{O}_{C}\) for some curve \(C\subseteq Z\) of degree \(dH^2\).

	Let \(\mathcal{O}_{Z}^r\to\mathcal{L}\to 0\) be a quotient in the Quot-scheme
	\(\Quot(\mathcal{O}_{Z}^r;0,dH,-d^2H^2/2)\).
	Let \(\mathcal{T}\subseteq \mathcal{L}\) be the maximal zero-dimensional (torsion) subsheaf,
	and set \(\mathcal{F}=\mathcal{L}/\mathcal{T}\).
	Then \(\mathcal{F}\) is a pure sheaf of dimension \(1\).

	Assume first that the support of \(\mathcal{F}\) is irreducible, say \(\Supp\mathcal{F}=nC\).
	By \cite[Proposition~5.10]{yuan2018motivic},
	the sheaf \(\mathcal{F}\) admits an \emph{upper filtration}
	\[
		0=\mathcal{F}^0\subsetneq \mathcal{F}^1\subsetneq \cdots \subsetneq \mathcal{F}^{m}=\mathcal{F}
	\]
	whose successive quotients \(\mathcal{R}_i\coloneqq\mathcal{F}^i/\mathcal{F}^{i-1}\) are
	sheaves on \(C\) of rank \(r_i\) (not necessarily torsion-free).
	These ranks satisfy \(\sum r_i=n\).
	For each \(i\geq 2\), the structure of \(\mathcal{F}\) induces a surjection
	\(\mathcal{R}_i(-C)\to\mathcal{R}_{i-1}\to 0\),
	and for every \(i\) there is also a surjection \(\mathcal{O}_{C}^r(-(m-i)C)\to\mathcal{R}_i\to 0\).

	Let \(\deg\) denote the degree on \(C\).
	Applying the slope inequality to the latter surjections yields \(\deg\mathcal{R}_i\geq -(m-i)r_iC^2\).
	Using additivity of Chern characters on \(C\), we compute
	\[
		\ch_2(\mathcal{F})=\sum_{i=1}^{m}\ch_2(\mathcal{R}_i)\geq \sum_{i=1}^m -\frac{r_i}{2}C^2 + \sum_{i=1}^m -(m-i)r_iC^2.
	\]
	Since \(C^2>0\) and \(\ch_2(\mathcal{F})\leq \ch_2(\mathcal{L})=-n^2C^2/2\),
	we have
	\[
		\frac{1}{2}n^2 = \frac{1}{2}(\sum r_i)^2\leq\sum\frac{r_i}{2}+(m-i)r_i.
	\]
	By \zcref{lem:main_ineq}, equality must hold everywhere,
	and hence \(r_i=1\) for all \(i\).
	In particular, \(\mathcal{T}=0\), and each \(\mathcal{R}_i\) is a rank-one torsion sheaf on \(C\).
	Moreover, the degree bound becomes an equality \(\deg\mathcal{R}_i=-(m-i)C^2\).
	By \zcref{lem:torsion_free_sheaf_pullback}, this implies \(\mathcal{R}_i\simeq \mathcal{O}_{C}(-(m-i)C)\).

	Using the filtration and the surjections,
	one constructs inductively a compatible system of extensions realizing \(\mathcal{F}\)
	as an iterated extension of the \(\mathcal{R}_i\).
	In particular, at the first stages one obtains a commutative diagram whose rows and columns are exact
	and whose maps are induced by the quotient \(\mathcal{O}_{nC}^{\oplus r}\twoheadrightarrow\mathcal{F}\).
	\[
		\begin{tikzcd}
			0 \ar[d] & 0 \ar[d] & 0 \ar[d] \\
			\mathcal{O}_{C}(-(n-1)C) \ar[r] \ar[d] & \mathcal{O}^r_{C}(-(n-1)C) \ar[r] \ar[d] &
			\mathcal{F}^1 \ar[d] \\
			\mathcal{O}_{2C}(-(n-2)C) \ar[r] \ar[d] & \mathcal{O}^r_{2C}(-(n-2)C) \ar[r] \ar[d] &
			\mathcal{F}^2 \ar[d] \\
			\mathcal{O}_{C}(-(n-2)C) \ar[r] \ar[d] & \mathcal{O}^r_{C}(-(n-2)C) \ar[r] \ar[d] &
			\mathcal{R}_2 \ar[d] \\
			0 & 0 & 0
		\end{tikzcd}
	\]
	By choosing the first column appropriately,
	one may arrange that the compositions in the top and third rows are the identity.
	It follows that \(\mathcal{F}^2\simeq \mathcal{O}_{2C}(-(n-2)C)\).
	By induction, we can show that \(\mathcal{L}=\mathcal{F}\simeq\mathcal{O}_{nC}\).

	If \(\Supp\mathcal{F}\) has more than one irreducible component,
	the same argument applies to each component separately, using additivity of Chern characters.
	In particular, \(\mathcal{L}\) is pure, and the given Chern character forces \(\mathcal{L}\simeq\mathcal{O}_{C}\)
	for some curve \(C\) of degree \(dH^2\).

	Therefore, every quotient in \(\Quot(\mathcal{O}_{Z}^r;0,dH,-d^2H^2/2)\)
	is of the form \(\mathcal{O}_{C}\),
	and in particular is pure.
	By \zcref{rmk:locally_free_equivalence} and the construction in \zcref{prop:relation_quot_scheme},
	this Quot-scheme coincides with the moduli space \(M_Z(r,dH,d^2H^2/2)\).
\end{proof}

\begin{proposition}\label{prop:hc_morphism_existence}
	There is a well-defined Hilbert-Chow morphism
	\[
		\sigma\colon M_Z(r,dH,d^2H^2/2) \longrightarrow \CDiv^d(Z)
	\]
	sending \([\mathcal{E},s]\) to the Fitting support of \(\mathcal{Q}\).
\end{proposition}

\begin{proof}
	It suffices to construct the morphism for the Quot-scheme
	\(\mathfrak{Q}\coloneqq\Quot(\mathcal{O}_{Z}^r;0,dH,-d^2/2)\)
	and then use the previous proposition.
	Let
	\[
		0\longrightarrow \mathcal{K}_{\mathfrak{Q}}\xlongrightarrow{\alpha} \mathcal{O}_{\mathfrak{Q}}^r\longrightarrow
		\mathcal{L}_{\mathfrak{Q}}\longrightarrow 0
	\]
	be the universal quotient on \(\mathfrak{Q}\times Z\).
	It is known that \(\mathcal{K}_{\mathfrak{Q}}\) is locally free of rank \(r\).
	Then determinant \(\det\alpha\) is then a section of the line bundle \(\det\mathcal{K}_{\mathfrak{Q}}^{-1}\),
	and its zero divisor defines a relative effective Cartier divisor on \(\mathfrak{Q}\times Z\) of degree \(dH^2\).
	This induces a morphism \[
		\sigma\colon\mathfrak{Q}\longrightarrow\CDiv^d(Z)
	\]
	which coincides with the Fitting support construction on closed points.
\end{proof}

\begin{theorem}\label{thm:hc_fibre_pic_one}
	The Hilbert-Chow morphism
	\[
		\sigma\colon M_Z(r,dH,d^2/2)\longrightarrow\CDiv^d(Z)
	\]
	is a \(\mathbb{P}^{r-1}\)-bundle.
\end{theorem}

\begin{proof}
	By the preceding proposition, any point of \(M_Z(r,dH,d^2/2)\)
	lying over a divisor \(C\in\CDiv^d(Z)\) corresponds to a quotient \(\mathcal{O}_{Z}^r\to\mathcal{O}_{C}\).
	Tensoring with \(\mathcal{O}_{C}\) we obtain
	\[
		\mathcal{O}_{C}^r\longrightarrow \mathcal{O}_{C}
	\]
	and conversely any such surjection arises in this way.
	Thus the fibre of \(\sigma\) over \([C]\) is naturally identified with
	the projective space \(\mathbb{P}(\Hom(\mathcal{O}_{C}^r,\mathcal{O}_{C}))\simeq \mathbb{P}^{r-1}\).
	These identifications vary algebraically with \(C\),
	so \(\sigma\) is a \(\mathbb{P}^{r-1}\)-bundle.
\end{proof}

\begin{proof}[Proof of \zcref{thm:moduli_space_picard_one}]
	It follows from \zcref{prop:moduli_quot_isomorphism,prop:hc_morphism_existence,thm:hc_fibre_pic_one}.
\end{proof}

\subsection*{Geometric interpretation in rank 2}

When \(r=2\), we may apply the geometric treatment of \zcref{sec:geometric_treatment}
to describe the fibres of the Hilbert-Chow morphism
\[
	\sigma\colon M_Z(2,dH,d^2H^2/2) \longrightarrow \CDiv^d(Z).
\]
In fact, the same method applies for any rank \(r\).

\begin{lemma}
	Let \(C\subseteq Z\) be a curve on a normal projective surface.
	If \(\pi\colon Z'\to Z\) is a birational morphism from another normal projective surface and
	\[
		\pi^{*}C=\widetilde{C}+\sum_j n_jE_j,
	\]
	where \(\widetilde{C}\) is the strict transform and  \(E_j\) are \(\pi\)-exceptional curves.
	Then
	\[
		C^2-\widetilde{C}^2=\widetilde{C}\cdot\sum n_jE_j.
	\]
\end{lemma}

\begin{proof}
	For each \(j\), we have \(\pi^{*}C\cdot E_j=0\), hence
	\[
		\widetilde{C}\cdot E_j = - \sum n_iE_i\cdot E_j
	\]
	and
	\[
		\widetilde{C}\cdot\sum n_jE_j = - \sum_jn_j(\sum_i n_iE_i\cdot E_j) = -(\sum n_jE_j)^2.
	\]
	Then
	\begin{align*}
		C^2 = (\pi^{*}C)^2 & = (\widetilde{C}+\sum n_jE_j)^2 \\
		& = \widetilde{C}^2 + 2\widetilde{C}\cdot\sum n_jE_j + (\sum n_jE_j)^2 \\
		& = \widetilde{C}^2 + \widetilde{C}\cdot\sum n_jE_j
	\end{align*}
	as desired.
\end{proof}

Let \(\pi\colon Z'\to Z\) be the minimal log resolution,
such that \(\pi^{*}Q+\Exc(\pi)\) has simple normal crossing.
As explained in \zcref{sec:geometric_treatment},
in the backward transform there is some freedom in the choice of blowup centres above \(\pi\)-exceptional curves,
but only the choices of blowup centres above the strict transform of \(Q\) affect the resulting associated model.
Thus we may first perform the Maruyama transforms along the exceptional components,
and only then along the strict transform of \(Q\).

By the previous lemma and the analysis in \zcref{sec:geometric_treatment},
for each component \(\widetilde{C}\) of the strict transform of \(Q\) the self-intersection number \(a\)
appears in \zcref{prop:self_intersection_blowup_curve} is at most \(-\widetilde{C}\cdot\sum_{j}n_jE_j\),
while the self-intersection number \(b\) equals \(\widetilde{C}^2\).
So \(b-a\geq C^2>0\) and in \zcref{cor:ruled_surface_inv},
we are always in the case where there is a unique negative section to choose as the blowup centre,
and the effective choices of blowup centres used in the backward transform form a projective line.
This gives a geometric interpretation of the fibres of the Hilbert-Chow morphism as \(\mathbb{P}^1\)'s
in the rank \(2\) case, compatible with \zcref{thm:hc_fibre_pic_one}.

\section{Negative curves}\label{sec:negative_curves}

In this section we fix a smooth projective surface \(Z\) containing a smooth rational curve \(C\)
with self-intersection \(-d\) for some \(d>0\).
We study the moduli space of sheaf stable pairs \([\mathcal{E},s]\)
of rank \(2\) on \(Z\) whose cokernel supported on \(C\).
By embedding this moduli space into a suitable Grassmannian, we describe its irreducible components.

We focus on \(M_Z(2,2C,-2d)\),
the moduli space of sheaf stable pairs \([\mathcal{E},s]\) with \(\ch(\mathcal{E})=(2,2C,-2d)\).
Either \(\mathcal{E}\) is locally free,
or there is an inclusion \(\mathcal{E}\hookrightarrow\mathcal{E}^{\vee\vee}\) whose cokernel is a skyscraper sheaf.
In the locally free case one may use either the algebraic (\zcref{lem:quotient_to_o2c}) or
the geometric (\zcref{rmk:geometric_inter_2C}) methods to identify the corresponding component;
it is isomorphic to an \(\mathbb{A}^{d+1}\)-bundle over \(\mathbb{P}^1\).
In the non‑locally free case the moduli problem reduces to a Quot-scheme on the curve \(C\).

\begin{remark}
	If the cokernel subscheme is reduced,
	then the moduli space \(M_Z(r,C,-d/2)\) is isomorphic to \(\mathbb{P}^1\).
	Indeed, arguing as in the proof of \zcref{prop:moduli_quot_isomorphism},
	one shows that any such sheaf stable pair has \(\mathcal{E}\) locally free.
	Hence \(M_Z(r,C,-d/2)\) identifies with the Quot-scheme \(\Quot(\mathcal{O}_{Z}^r;0,C,d/2)\)
	and the latter is \(\mathbb{P}^{r-1}\) by the same argument as in \zcref{thm:hc_fibre_pic_one}
	(or via the geometric interpretation of \zcref{sec:geometric_treatment}).
\end{remark}

\medskip

Let \(\mathcal{E},s\) be a sheaf stable pair in \(M_Z(2,2C,-2d)\).
Applying the dual functor \(\SHom(-,\mathcal{O}_{Z})\) to the short exact sequence
\[
	0\longrightarrow \mathcal{O}_Z^2 \xlongrightarrow{s} \mathcal{E}\longrightarrow \mathcal{Q}\longrightarrow 0,
\]
where \(\mathcal{Q}=\coker(s)\), we obtain
\begin{equation}\label{eq:surjective_F}
	0\longrightarrow \mathcal{E}^{\vee} \xlongrightarrow{s^{\vee}} \mathcal{O}_Z^2
	\longrightarrow\mathcal{F}\longrightarrow 0,
\end{equation}
where \(\mathcal{F}\) is the image of the natural morphism \(\mathcal{O}_Z^2\to \SExt^1(\mathcal{Q},\mathcal{O}_{Z})\)
and \(\mathcal{Z}_f(\mathcal{F})=\mathcal{Z}_f(\mathcal{Q})=2C\).

\begin{lemma}\label{lem:F_two_cases}
	The sheaf \(\mathcal{F}\) is isomorphic to either \(\mathcal{O}_{2C}\) or \(\mathcal{O}_C^2\).
\end{lemma}

\begin{proof}
	Since there is a surjection \(\mathcal{O}_{Z}^2\to \mathcal{F}\), we have \(\rk \mathcal{F}\leq 2\).
	So either \(\rk \mathcal{F}=2\) and \( \mathcal{Z}_a(\mathcal{F})=C\) or
	\(\rk \mathcal{F}=1\) and \(\mathcal{Z}_a(\mathcal{F})=2C\).
	Note also that \(\mathcal{F}\) is pure as \(\mathcal{E}^{\vee}\) is locally free.
	It then follows that \(\mathcal{F}\) is locally free \cite[Lemma~2.5]{yang2003coherent}.

	Consider the trivial projective bundle \(\pi\colon X=\mathbb{P}(\mathcal{O}_Z^2)\to Z\).
	By \zcref{eq:surjective_F}, \(Y\coloneqq\mathbb{P}(\mathcal{F})\) can be viewed as a subscheme of \(X\)
	and also a projective subbundle of \(X_D\simeq D\times\mathbb{P}^1\)
	where \(D=\mathcal{Z}_a(\mathcal{F})\).
	So either \(\mathcal{Z}_a(\mathcal{F})=C\), \(Y=X_C\) and \(\mathcal{F}\simeq \mathcal{O}_{C}^2\)
	or \(\mathcal{Z}_a(\mathcal{F})=2C\) and \(Y\) is a cross-section of \(X_{2C}\).
	In the latter case, pushing forward the exact sequence
	\[
		0\longrightarrow \mathcal{I}_Y\otimes\mathcal{O}_{X}(1)\longrightarrow \mathcal{O}_{X}(1)\longrightarrow
		\mathcal{O}_{X}(1)\otimes\mathcal{O}_{Y}\longrightarrow 0
	\]
	gives
	\[
		0\longrightarrow \pi_*(\mathcal{I}_Y\otimes\mathcal{O}_{X})\longrightarrow \mathcal{O}_Z^2\longrightarrow
		\mathcal{F}\longrightarrow 0
	\]
	(\cite[Proof of Proposition~1.6]{maruyama1982elementary}),
	which is \eqref{eq:surjective_F}.
	Since \(\mathcal{O}_{X}(1)\) is free, \(\mathcal{O}_{X}(1)\otimes\mathcal{O}_{Y}=\mathcal{O}_{Y}\)
	and consequently \(\mathcal{F}=\pi_*\mathcal{O}_{Y}\simeq\mathcal{O}_{2C}\).
\end{proof}

By the correspondence \textbf{from a sheaf stable pair to a quotient} in \zcref{prop:relation_quot_scheme}
and using the notations there, we obtain an short exact sequence
\begin{equation}\label{eq:torsion_free_exact}
	0\to \mathcal{T}\to \mathcal{L} \to \mathcal{F} \to 0,
\end{equation}
where \(\mathcal{T}\) is the maximal torsion subsheaf of \(\mathcal{L}\) and \(\ch(\mathcal{L})=(0,2C,2d)\).

\begin{lemma}\label{lem:E_locally_free}
	The sheaf \(\mathcal{E}\) is locally free if and only if \(\mathcal{F}\simeq \mathcal{O}_{2C}\).
	In that situation, \(\mathcal{L}=\mathcal{F}\simeq \mathcal{O}_{2C}\).
\end{lemma}

\begin{proof}
	By \zcref{rmk:locally_free_equivalence},
	\(\mathcal{E}\) is locally free precisely when the sheaf \(\mathcal{L}\) is pure,
	i.e., when \(\mathcal{T}=0\).
	\zcref{eq:torsion_free_exact} gives \(\ch(\mathcal{L})=\ch(\mathcal{F})+\ch(\mathcal{T})\).
	Since \(\ch(\mathcal{O}_{2C})=(0,2C,2d)\) and \(\ch(\mathcal{O}_{C}^{2})=(0,2C,d)\),
	the only way to have \(\mathcal{T}=0\) is \(\mathcal{F}\simeq\mathcal{O}_{2C}\).
	Conversely, if \(\mathcal{F}\simeq\mathcal{O}_{2C}\) then \(\mathcal{T}=0\) and \(\mathcal{L}=\mathcal{F}\).
\end{proof}

\begin{lemma}\label{lem:non_pure_case}
	If \(\mathcal{E}\) is not locally free, then \(\mathcal{F}\simeq\mathcal{O}_{C}^2\) and \(h^0(Z,\mathcal{T})=d\).
\end{lemma}

\begin{proof}
	By \zcref{lem:F_two_cases,lem:E_locally_free},
	the only remaining possibility is \(\mathcal{F}\simeq \mathcal{O}_C^2\).
	Now \(\ch(\mathcal{F})=(0,2C,d)\), so \(\ch_2(\mathcal{T})=d\).
	As \(\mathcal{T}\) is a skyscraper sheaf,
	we have \(h^{0}(Z,\mathcal{T})=\ch_2(\mathcal{T})=d\).
\end{proof}

We next describe the relationship between moduli space \(M(2,2C,-2d)\) and
the Quot-scheme \(\Quot(\mathcal{O}_{Z}^2;0,2C,2d)\) appearing in \zcref{prop:relation_quot_scheme}.

\begin{definition}
	We denote some loci in the moduli space as follows.
	\begin{itemize}[leftmargin=*]
		\item \(M_f\subseteq M(2,2C,-2d)\), the subscheme parametrising sheaf stable pairs \(\mathcal{E},s\)
			with \(\mathcal{E}\) locally free.
		\item \(M_f'=M_Z(2,2C,-2d)\setminus M_f\), the complement of \(M_f\),
			parametrising sheaf stable pairs \(\mathcal{E},s\) with \(\mathcal{E}\) non-locally free.
		\item \(\Theta\subseteq\Quot(\mathcal{O}_{Z}^2,0,2C,2d)\), the subscheme parametrising quotients
			\(\mathcal{O}_Z^2\to\mathcal{L}\) for which \(\Supp(\mathcal{T})\subseteq C\),
			where \(\mathcal{T}\) is the maximal torsion subsheaf of \(\mathcal{L}\).
		\item \(\Theta_d\subseteq\Theta\), the locus quotients with \(\mathcal{T}\neq 0\) supports on \(C\).
		\item \(\Theta_p\subseteq\Theta\), the locus of quotients with \(\mathcal{L}\) pure,
			which is also the complement of \(\Theta_d\) in \(\Theta\).
	\end{itemize}
\end{definition}

\begin{theorem}\label{thm:M_f_theta_p}
	The subscheme \(M_f\) is isomorphic to \(\Theta_p\).
	The subscheme \(M_f'\) is isomorphic to \(\Theta_d\), which is also isomorphic to \(\Quot(\mathcal{O}_C^2, d)\).
	Consequently, the moduli space \(M(2,2C,-2d)\) is isomorphic to \(\Theta\).
\end{theorem}

To show \zcref{thm:M_f_theta_p}, we need the following results.

\begin{proposition}\label{prop:M_f'_and_quot}
	The subscheme \(M_f'\) is isomorphic to the Quote-scheme \(\Quot(\mathcal{O}_C^2, d)\).
\end{proposition}

\begin{proof}
	Let \(\mathcal{E},s\) be a sheaf stable pair in \(M(2,2C,-2d)\) with \(\mathcal{E}\) not locally free.
	Then we have a commutative diagram with exact rows
	\[
		\begin{tikzcd}
			0 \ar[r] & \mathcal{O}_Z^r \ar[r,"s"] \ar[d,"\cong"] & \mathcal{E} \ar[r] \ar[d,hookrightarrow] & \mathcal{Q}
			\ar[r] \ar[d] & 0 \\
			0 \ar[r] & \mathcal{O}_Z^r \ar[r,"s^{\vee\vee}"] & \mathcal{E}^{\vee\vee} \ar[r] & \mathcal{Q}' \ar[r] & 0,
		\end{tikzcd}
	\]
	where \(\mathcal{Q}'=\coker(s^{\vee\vee})\).
	By \zcref{lem:non_pure_case}, we have \(\mathcal{F}\simeq \mathcal{O}_C^2\).
	From the exact sequence \eqref{eq:surjective_F} it follows that \(\mathcal{E}^{\vee}\simeq\mathcal{O}_Z^2(-C)\);
	hence \(\mathcal{E}^{\vee\vee}\simeq \mathcal{O}_Z^2(C)\) and \(\mathcal{Q}'\simeq\mathcal{O}_C^2(C)\).
	The snake lemma shows that the middle and the right arrows have the same cokernel, say \(\mathcal{T}'\).
	Thus we obtain a surjection \(\mathcal{O}_C(C)\to \mathcal{T}'\to 0\).

	Conversely, given such a quotient \(\mathcal{O}_{C}(C)\twoheadrightarrow \mathcal{T}'\), define
	\[
		\mathcal{E}\coloneqq \ker(\mathcal{O}_{Z}^2(C)\twoheadrightarrow \mathcal{O}_{C}^2(C)\twoheadrightarrow
		\mathcal{T}').
	\]
	Since the composition
	\[
		\mathcal{O}_{Z}^2\longrightarrow \mathcal{O}_{Z}^2(C)\longrightarrow \mathcal{O}_{C}^2(C)\longrightarrow \mathcal{T}'
	\]
	is zero, we obtain an injection \(\mathcal{O}_Z^2 \to\mathcal{E}\)
	whose cokernel is a subsheaf of \(\mathcal{O}_{C}^2(C)\).
	This gives a sheaf stable pair in \(M(2,2c,-2d)\).

	By \zcref{prop:relation_quot_scheme},
	we have \(\mathcal{T}'=\SExt^2(\mathcal{T},\omega_Z)\).
	In particular, viewed as a sheaf on \(C\),
	\(\mathcal{T}'\) has length \(h^0(Z,\mathcal{T}')=d\).
	Moreover, quotients \(\mathcal{O}_C^2(C)\to\mathcal{T}'\to 0\) are in bijection with
	quotients \(\mathcal{O}_C^2\to\mathcal{T}'\to 0\) in \(\Quot(\mathcal{O}_C^2, d)\).
	Hence \(M_f'\simeq\Quot(\mathcal{O}_{C}^2,d)\).
\end{proof}

\begin{lemma}\label{lem:F_equals_oc2}
	Let \(\mathcal{O}_Z^2\to\mathcal{L}\) be a point in \(\Quot(\mathcal{O}_{Z}^2,0,2C,2d)\).
	Denote by \(\mathcal{T}\subseteq \mathcal{L}\) the maximal torsion subsheaf
	and set \(\mathcal{F}\coloneqq\mathcal{L}/\mathcal{T}\).
	If \(\mathcal{T}\neq 0\) then \(\mathcal{F}\simeq\mathcal{O}_C^2\).
	Consequently the support of \(\mathcal{T}\) has length \(d\).
\end{lemma}

\begin{proof}
	Note that \(\mathcal{F}\) is a pure sheaf \cite[Definition~1.1.4]{huybrechts2010geometry}.
	Because \(\mathcal{T}\neq 0\) is a skyscraper sheaf,
	we have \(\ch_1(\mathcal{F})=\ch_1(\mathcal{L})=2C\).
	Moreover, \(h^0(\mathcal{T})+\ch_2(\mathcal{F})=\ch_2(\mathcal{L})=2d\).
	Since \(\mathcal{T}\neq 0\), we get \(h^0(\mathcal{T})>0\).
	Hence \(\ch_2(\mathcal{F})<2d\).
	If \(\mathcal{Z}_a(\mathcal{F})=2C\),
	since there is a surjection \(\mathcal{O}_{2C}^2\to \mathcal{F}\), \(\deg(\mathcal{F})\geq 0\),
	then \(\ch_2(\mathcal{F})=-\frac{1}{2}(2C)^2+\deg(\mathcal{F})\geq 2d\),
	which leads to contradiction.
	Hence \(\mathcal{Z}_a(\mathcal{F})=C\) and \(\mathcal{F}\) is locally free of rank \(2\) on \(C\).
	Since the surjective map \(\mathcal{O}_Z^2\to \mathcal{F}\to 0\),
	we obtain \(\mathcal{F}=\mathcal{O}_C^2\).
	In this case, \(\ch_2(\mathcal{F})=d\), which implies that \(\mathcal{T}\) has length \(d\).
\end{proof}

\begin{proposition}\label{prop:theta_d_and_quot}
	The subscheme \(\Theta_d\) is isomorphic to \(\Quot(\mathcal{O}_C^2,d)\).
\end{proposition}

\begin{proof}
	Let \(q\colon\mathcal{O}_Z^2\to\mathcal{L}\) be a point in \(\Theta_d\).
	By definition the maximal torsion subsheaf \(\mathcal{T}\subseteq \mathcal{L}\) is a skyscraper sheaf
	supported on \(C\) of length \(d\).
	Set \(\mathcal{K}\coloneqq \ker q\) and let \(\mathcal{F}\coloneqq \mathcal{L}/\mathcal{T}\).
	We have a commutative diagram with exact rows
	\[
		\begin{tikzcd}
			0 \ar[r] & \mathcal{K} \ar[r] \ar[d] & \mathcal{O}_{Z}^2 \ar[r] \ar[d,"="] & \mathcal{L}
			\ar[r] \ar[d] & 0 \\
			0 \ar[r] & \mathcal{K}' \ar[r] & \mathcal{O}_{Z}^2 \ar[r] & \mathcal{F} \ar[r] & 0.
		\end{tikzcd}
	\]
	The snake lemma that \(\mathcal{K}\hookrightarrow\mathcal{K}'\) is injective
	with cokernel isomorphic to \(\mathcal{T}\).
	By \zcref{lem:F_equals_oc2},
	\(\mathcal{F}\simeq\mathcal{O}_C^2\) and \(\mathcal{T}\) has length \(d\).
	Thus \(\mathcal{K}'\simeq\mathcal{O}_Z^2(-C)\).
	Then we obtain a surjection \(\mathcal{O}_Z^2(-C)\to\mathcal{T}\).

	Conversely, let \(p\colon\mathcal{O}_{Z}^2(-C)\to\mathcal{T}\)
	be a quotient with \(\Supp(\mathcal{T})\subseteq C\) and \(\ell(\mathcal{T})=d\).
	Let \(\mathcal{K}\coloneqq\ker p\).
	The injection \(\mathcal{K}\hookrightarrow \mathcal{O}_{Z}^2\) followed by the natural map
	\(\mathcal{O}_{Z}^2(-C)\to \mathcal{O}_{Z}^2\) yields an short exact sequence
	\[
		0\longrightarrow \mathcal{K}\longrightarrow \mathcal{O}_{Z}^2\longrightarrow \mathcal{L}\longrightarrow 0.
	\]
	The maximal torsion subsheaf of \(\mathcal{L}\) is isomorphic to \(\mathcal{T}\).
	Hence \(q\colon \mathcal{O}_{Z}^2\to \mathcal{L}\) is a point of \(\Theta_d\).

	Thus \(\Theta_d\) is canonically identified with the space of quotients \(O_{Z}^2(-C)\to\mathcal{T}\)
	with \(\Supp(\mathcal{T})\subseteq C\) and \(\ell(\mathcal{T})=d\).
	Finally, because \(\mathcal{T}\) is supported on \(C\),the restriction map
	\[
		\Hom(\mathcal{O}_{Z}^2(-C),\mathcal{T}) \longrightarrow \Hom(\mathcal{O}_C^2,\mathcal{T})
	\]
	is an isomorphism.
	Therefore, \(\Theta_d\) is isomorphic to the Quot-scheme \(\Quot(\mathcal{O}_{C}^2,d)\).
\end{proof}

\medskip

\begin{proof}[Proof of \zcref{thm:M_f_theta_p}]
	The first isomorphism is an immediate consequence of \zcref{prop:relation_quot_scheme}
	together with \zcref{rmk:locally_free_equivalence}.
	For the second isomorphism, \zcref{prop:relation_quot_scheme} gives \(M_f'\subseteq \Theta_d\).
	Combining this with \zcref{prop:M_f'_and_quot,prop:theta_d_and_quot} yields the required isomorphism.
\end{proof}

\subsection{Study \(M_f\)}

Since \(M_{f}\simeq \Theta_p\) by \zcref{thm:M_f_theta_p},
in order to parametrise all the sheaf stable pair \(\mathcal{E},s\) in \(M_Z(2,2C,-2d)\)
with \(\mathcal{E}\) locally free,
we only need to parametrise all the quotients \(\mathcal{O}_Z^2\to\mathcal{L}\to 0\) such that \(\mathcal{L}\) is pure.
Then by \zcref{lem:E_locally_free},
this is equivalent to parametrize all the quotients \(\mathcal{O}_Z^2\to\mathcal{O}_{2C} \to 0\).

\begin{lemma}\label{lem:global_sections_O2C}
	Let \(C\subseteq Z\) be a smooth rational curve with \(C^2=-d\).
	Then the ring of global sections of the structure sheaf \(\mathcal{O}_{2C}\) of the double curve \(2C\)
	is the \(k\)-algebra
	\[
		R\coloneqq k[x_0,\dots,x_d]/(x_0,\dots,x_d)^2.
	\]
\end{lemma}

\begin{proof}
	Because \(C\simeq\mathbb{P}^1\) and \(C^2=-d\),
	there is a short exact sequence \cite[p.62, (4)]{barth2004compact}
	\[
		0\longrightarrow \mathcal{O}_{C}(d)\longrightarrow \mathcal{O}_{2C}\longrightarrow \mathcal{O}_C\longrightarrow 0.
	\]
	We may view \(2C\) as an infinitesimal extension of \(C\),
	which is trivial since
	\[
		H^1(C,\SHom(\Omega_C,\mathcal{O}_C(d)))=0.
	\]
	Hence \(\mathcal{O}_{2C}\simeq \mathcal{O}_C\oplus\mathcal{O}_C(d)\) as sheaves of \(\mathcal{O}_C\)-algebras,
	with multiplication given by
	\[
		(a_1\oplus s_1)\cdot (a_2\oplus s_2)=(a_1a_2\oplus (a_1s_2+a_2s_1))
	\]
	\cite[III, Exercise~4.10]{hartshorne1977algebraic}.
	Taking global sections yields
	\[
		H^{0}(Z,\mathcal{O}_{2C})\simeq H^{0}(C,\mathcal{O}_{C})\oplus H^{0}(C,\mathcal{O}_{C}(d))
		\simeq k[x_0,\dots,x_d]/(x_0,\dots,x_d)^2.
	\]
	This identifies the desired algebra \(R\).
\end{proof}

\begin{lemma}\label{lem:quotient_to_o2c}
	All the surjective morphisms
	\[
		\mathcal{O}_{Z}^2\to \mathcal{O}_{2C}\to 0
	\]
	are parametrised by the total space of the vector bundle \(\mathcal{O}_{\mathbb{P}^1}(-2)^{\oplus (d+1)}\).
	Equivalently, this \(\mathbb{A}^{d+1}\)-bundle over \(\mathbb{P}^1\) is the smooth locus of
	\(V(xz+y^2)\subseteq \mathbb{P}^{d+3}\).
\end{lemma}

\begin{proof}
	Since the support of \(\mathcal{O}_{2C}\) is \(2C\),
	any surjection \(\mathcal{O}_{Z}^2\to\mathcal{O}_{2C}\to 0\) factors through
	\[
		\mathcal{O}_{2C}^2\longrightarrow \mathcal{O}_{2C}\longrightarrow 0.
	\]
	Thus we may work with surjections from \(\mathcal{O}_{2C}^2\).

	A quotient \(\mathcal{O}_{2C}^2\to\mathcal{O}_{2C}\) is given by a pair \((e,h)\in\mathcal{O}_{2C}^2\)
	whose components generate \(\mathcal{O}_{2C}\) as an \(\mathcal{O}_{2C}\)-module;
	equivalently, at least one of \(e,h\) is invertible.
	The kernel is the submodule generated by \((-h,e)\).

	If \(e\) is invertible, then
	\[
		\langle(-h,e)\rangle = \langle(-e^{-1}h,1)\rangle
	\]
	so the data of the surjection is encoded by the element
	\begin{equation}\label{eq:generator_1}
		e^{-1}h=a_1+\sum_{i=0}^d b_ix_i \in R,
	\end{equation}
	which is uniquely determined by the coordinates	\((a_1,b_0,\dots,b_d)\in\mathbb{A}^1\times\mathbb{A}^{d+1}\).
	Similarly, if \(h\) is invertible,
	we obtain coordinates \((a_2,c_0,\dots,c_d)\) via
	\[
		-h^{-1}e=a_2+\sum_{i=0}^d c_ix_i \in R.
	\]
	When both \(e\) and \(h\) are invertible, the two descriptions are related by
	\[
		(a_{1},b_{0},\dots,b_{d})\longmapsto
		\Bigl(\frac{1}{a_{1}},-\frac{b_{0}}{a_{1}^{2}},\dots,-\frac{b_{d}}{a_{1}^{2}}\Bigr),
	\]
	which is precisely the transition function of \(\mathcal{O}_{\mathbb{P}^1}(-2)^{\oplus(d+1)}\)
	on the standard cover of \(\mathbb{P}^1\).
	Hence the moduli space of surjections is the total space of this vector bundle.

	To identify this space with the smooth locus of \(V(xz+y^{2})\subseteq\mathbb{P}^{d+3}\),
	let \((x:y:z:u_{0}:\dots:u_{d})\) be homogeneous coordinates on \(\mathbb{P}^{d+3}\).
	On the chart \(x\neq0\) we set \(x=1\) and obtain coordinates
	\((y,u_{0},\dots,u_{d})\in\mathbb{A}^1\times\mathbb{A}^{d+1}\).
	On the chart \(z\neq0\) we set \(z=1\) and obtain coordinates
	\((-y,u_{0},\dots,u_{d})\in\mathbb{A}^1\times\mathbb{A}^{d+1}\).
	The intersection of these charts is given by \(x\neq 0\), \(z\neq 0\), \(y\neq 0\),
	and the transition between the two coordinate systems is
	\[
		(y,u_{0},\dots,u_{d})\longmapsto
		\Bigl(\frac{1}{y},-\frac{u_{0}}{y^{2}},\dots,-\frac{u_{d}}{y^{2}}\Bigr),
	\]
	again the transition function of \(\mathcal{O}_{\mathbb{P}^1}(-2)^{\oplus(d+1)}\).
	The singular locus of \(V(xz+y^{2})\) is the linear subspace \(\{x=y=z=0\}\cong\mathbb{P}^d\).
	Thus the smooth locus of \(V(xz+y^{2})\) is precisely the total space of
	\(\mathcal{O}_{\mathbb{P}^1}(-2)^{\oplus(d+1)}\).
\end{proof}

\begin{remark}[Geometric interpretation]\label{rmk:geometric_inter_2C}
	For sheaf stable pairs that are locally free,
	we can use the geometric treatment in \zcref{eq:geometric_interpretation}
	to describe concretely the fibres of the Hilbert-Chow morphism
	\[
		\sigma\colon M_Z(2,2C,-2d)\longrightarrow \CDiv(Z), \quad [\mathcal{E},s]\longmapsto Q
	\]
	where \(Q\) is the Fitting support the cokernel sheaf \(\mathcal{Q}\coloneqq\coker(s)\).
	In the present situation \(Q=2C\);
	thus \(Q\) consists of a single curve \(C\) with multiplicity two and \((Z,Q_{\red})\) is log smooth.
	By \zcref{thm:existence_stable_minimal_model}, all these sheaf stable pairs share
	the same associated stable minimal model
	\begin{equation}\label{eq:stable_minimal_model_2C}
		(X',B'),A' \xlongrightarrow{f'} Z,
	\end{equation}
	with \(X'\simeq Z\times\mathbb{P}^1\).
	It follows that both the maps \(\pi\) and \(\theta\) in the diagram \eqref{eq:geometric_interpretation}
	are identities.
	We may perform Maruyama's elementary transformation twice on the model \eqref{eq:stable_minimal_model_2C}.

	Pick a point \(p\in\mathbb{P}^1\) and take the fibre \(X'_p\) of the projection \(X'\to\mathbb{P}^1\) over \(p\).
	Denote by \(C'\) the intersection of \(X'_C=C\times\mathbb{P}^1\) and \(X'_p\).
	Then \(C'\) has self-intersection \(0\) in \(X'_C\).
	Blow up \(X'\) along \(C'\) and let \(C_1\) be the intersection of the exceptional divisor \(E\)
	with the strict transform of \(X'_C\).
	Then blow down the strict transform of \(X'_C\) to the curve \(C_1\);
	we continue to denote the strict transform of \(E\) by \(E\).
	Since \(C\) has self-intersection \(-d\) in \(Z\),
	by \zcref{prop:self_intersection_blowup_curve} the self-intersection of \(C_1\) in \(E\) equals
	\((C_1)^2_E-d-0=-d<0\),
	so \(E\) is a Hirzebruch surface of degree \(d\) by \zcref{cor:ruled_surface_inv}.
	For the second Maruyama transformation we choose a curve \(C''\subseteq E\) disjoint from \(C_1\).
	All the choices are parametrized by \(\mathbb{A}^{d+1}\) by \zcref{lem:parameter_space_cross_sections}.
	Hence the fibre of \(\sigma\) over \(Q=2C\) is an \(\mathbb{A}^{d+1}\)-fibration over \(\mathbb{P}^1\),
	in agreement with the description obtained from the algebraic method in \zcref{lem:quotient_to_o2c}.
\end{remark}

\subsection{Study \(M_f'\)}

Since \(M_f'\simeq\Theta_{d}\) (\zcref{thm:M_f_theta_p}),
any sheaf stable pair \([\mathcal{E},s]\in M_Z(2,2C,-2d)\) with \(\mathcal{E}\) not locally free
is determined by a quotient
\[
	[\mathcal{O}_{Z}^2\longrightarrow \mathcal{L}]\in\Quot(\mathcal{O}_{Z}^2,0,2C,2d)
\]
for which the maximal torsion subsheaf \(\mathcal{T}\subseteq\mathcal{L}\) has length \(d\) and
is supported on the curve \(C\).
Because \(\mathcal{L}\) is supported on \(2C\), the quotient factors through \(\mathcal{O}_{2C}^2\):
\[
	\mathcal{O}_{Z}^2\longrightarrow \mathcal{O}_{2C}^2\longrightarrow \mathcal{L}.
\]
Set \(R=H^{0}(Z,\mathcal{O}_{2C})\simeq k[x_{0}, \dots, x_{d}]\) (cf.~\zcref{lem:global_sections_O2C})
and \(L=H^{0}(Z,\mathcal{L})\).
Then a quotient \(\mathcal{O}_{2C}\to\mathcal{L}\) is uniquely determined by a surjection of \(R\)-modules
\[
	R^2\longrightarrow L.
\]
However, not every such surjection arises from a point of \(\Theta_d\);
one must impose additional conditions ensuring that the maximal torsion subsheaf \(\mathcal{T}\subseteq\mathcal{L}\)
has length \(d\), and that the quotient \(\mathcal{F}\coloneqq \mathcal{L}/\mathcal{T}\)
is isomorphic to \(\mathcal{O}_{C}^2\).

\smallskip

The surface singularity viewpoint clarifies the structure of \(R\).
The curve \(C\subseteq Z\) is a rational curve with self-intersection \(C^2=-d\) on the smooth projective
surface \(Z\).
Locally, it can always be realised as the exceptional divisor of the blow up of the vertex
the projective cone over \(\mathbb{P}^1\subset\mathbb{P}^d\) of degree \(d\)
(cf.~\cite[Chapter V, Example 2.11.4]{hartshorne1977algebraic}).
Let \(\phi\colon Z\to Y\) contract \(C\) to the point \(P\in Y\), which is singular when \(d\geq 2\).
The inverse image of the maximal ideal \(\mathfrak{m}_P\subseteq \mathcal{O}_{Y,P}\)
satisfies \(\phi^{-1}\mathfrak{m}_P\cdot\mathcal{O}_{Z}=\mathcal{I}_C\).
Hence
\[
	\mathcal{O}_{2C}= \mathcal{O}_Z/\mathcal{I}^2_C=
	\phi^*\mathcal{O}_Y/(\phi^{-1}\mathfrak{m}_P\cdot\mathcal{O}_Z)^2=\phi^*(\mathcal{O}_{Y}/\mathfrak{m}_P^2).
\]
Taking global sections, and using
\[
	\mathcal{O}_{Y}/\mathfrak{m}_P^2=k[x_0,\dots,x_d]/(x_0,\dots,x_d)^2,
\]
we obtain
\[
	H^{0}(Z,\mathcal{O}_{2C})\simeq H^{0}(Y,\mathcal{O}_Y/\mathfrak{m}_P^2)\simeq k[x_0,\dots,x_d]/(x_0,\dots,x_d)^2.
\]
Moreover, \(\mathcal{L}=\phi^{*}\phi_*\mathcal{L}\) for any quotient \(\mathcal{O}_{Z}^2\to\mathcal{L}\) in \(\Theta\).

\begin{lemma}\label{lem:limiting_case}
	Assume
	\[
		L=k\oplus k[x_0,\dots,x_d]/(x_0,\dots,x_d)^2+(a_0x_0+\cdots+a_dx_d)
	\]
	for some \([a_0:\cdots:a_d]\in\mathbb{P}^d\).
	Then any quotient \(R^2\to L\) of \(R\)-modules corresponds to a point of \(\Theta_d\).
\end{lemma}

\begin{proof}
	We regard \(L\) as a skyscraper sheaf supported at \(P\in Y\).
	Then \(\mathcal{L}=\phi^{*}L=\mathcal{O}_{C}\oplus\mathcal{O}_{C,p_1,\dots,p_d}\),
	where \(O_{C,p_{1}, \dots, p_{d}}\) denotes the structure sheaf of the curve \(C\) endowed with
	(non-necessarily distinct) non-reduced points \(p_{1}, \dots, p_{d}\).

	It is clear that the maximal torsion subsheaf \(\mathcal{T}\) of \(\mathcal{L}\) is isomorphic
	to the structure sheaf of these \(d\) points and \(\mathcal{L}/\mathcal{T}\simeq\mathcal{O}_{C}^2\).
	Since \(\phi^{*}R=\mathcal{O}_{C}\) and \(\phi^{*}\) is right exact,
	any quotient \(R^2\to L\) induces a quotient
	\[
		\mathcal{O}_{Z}^2\longrightarrow \mathcal{O}_{C}^{2}\longrightarrow \mathcal{L},
	\]
	which determines a point of \(\Theta_d\).
\end{proof}

By \zcref{lem:limiting_case}, it is natural to single out the following locus in \(\Theta_d\).

\begin{definition}\label{def:limiting_case}
	We define \(\Gamma\subseteq\Theta_d\) to be the locus parametrising all quotients
	\(\mathcal{O}_{2C}^2\to \mathcal{L}\) such that the space of global sections of \(\mathcal{L}\) is of the form
	\[
		k\oplus k[x_0,\dots,x_d]/(x_0,\dots,x_d)^2+(a_0x_0+\cdots+a_dx_d).
	\]
	In other words, \(\Gamma\) parametrises all surjections
	\[
		\big(k[x_0,\dots,x_d]/(x_0,\dots,x_d)^2\big)^{\oplus2}\longrightarrow k\oplus
		k[x_0,\dots,x_d]/(x_0,\dots,x_d)^2+(a_0x_0+\cdots+a_dx_d)
	\]
	for some \((a_0:a_1:\dots:a_d)\in\mathbb{P}^d\).
\end{definition}

The structure of \(\Gamma\) is easy to describe.

\begin{lemma}\label{lem:structure_of_gamma}
	For any fixed point \((a_0:\dots:a_d)\in\mathbb{P}^d\), the surjections
	\[
		\pi\colon\big(k[x_0,\dots,x_d]/(x_0,\dots,x_d)^2\big)^{\oplus2}\longrightarrow k\oplus
		k[x_0,\dots,x_d]/(x_0,\dots,x_d)^2+(a_0x_0+\cdots+a_dx_d).
	\]
	are parametrised by \(\mathbb{P}^1\).
	Moreover, \(\Gamma\) is isomorphic to \(\mathbb{P}^1\times\mathbb{P}^d\).
\end{lemma}

\begin{proof}
	Write
	\[
		\pi((1,0))=\Big(u',u+\sum_{i=0}^{d}u_ix_i\Big), \quad \pi((0,1))=\Big(v',v+\sum_{i=0}^{d}v_ix_i\Big)
	\]
	and let
	\[
		(f,g)=\Big(\alpha+\sum_{i=0}^{d}\alpha_ix_i,\beta+\sum_{i=0}^{d}\beta_ix_i\Big) \in \ker\pi.
	\]
	Since the homomorphism \(\pi\) is surjective, the matrix
	\[
		\begin{pmatrix}
			u' & u \\
			v' & v
		\end{pmatrix}
	\]
	has rank \(2\).
	By assumption,
	\begin{align*}
		\pi(f,g)&=f\pi((1,0))+g\pi((0,1))\\
		&=\Big(u'\alpha, \alpha u+\alpha\sum_{i=0}^{d}u_ix_i+\sum_{i=0}^{d}u\alpha_ix_i\Big)+
		\Big(v'\beta, \beta v+\beta\sum_{i=0}^{d}v_ix_i+\sum_{i=0}^{d}v\beta_ix_i\Big)=0.
	\end{align*}
	It follows that \(\alpha=\beta=0\).
	The kernel of \(\pi\) has \(k\)-dimension \(d+2\),
	and is generated as a \(k\)-vector space by \((\sum_{i=0}^{d}a_ix_i,0)\),
	\((0,\sum_{i=0}^{d}a_ix_i)\) and \((vx_i,-ux_i)\) for \(i=0,\cdots,d\),
	For fixed \((a_0:\cdots:a_d)\in\mathbb{P}^d\), the only remaining choice is the pair \((u:v)\in\mathbb{P}^1\).
	Hence all the quotients \(\pi\) with prescribed \((a_{0}:\dots:a_{d})\) are parametrized by \(\mathbb{P}^1\)
	and varying \((a_{0}:\dots:a_{d})\) we obtain an isomorphism \(\Gamma\simeq\mathbb{P}^1\times\mathbb{P}^d\).
\end{proof}

In fact, \(\Gamma\) is precisely the ``boundary'' of \(M_f\).

\begin{lemma}\label{lem:limit_of_gamma}
	Let \(\overline{M_f}\) denote the closure of \(M_f\) in \(M_Z(2,2C,-2d)\).
	Then \(\overline{M_f}\setminus M_f=\Gamma\).
\end{lemma}

\begin{proof}
	Points of \(M_f\) correspond to quotients \(\mathcal{O}_{Z}^2\to \mathcal{O}_{2C}\to 0\) described in
	\zcref{lem:quotient_to_o2c}.
	Such a quotient is determined by the choice of global sections \(e,h\in H^0(Z,\mathcal{O}_{2C})\).
	By \eqref{eq:generator_1} if \(e\) is invertible then the kernel is generated, as an \(R\)-module, by
	\[
		\Big(a_1+\sum_{i=0}^d{b_ix_i},1\Big)
	\]
	for some \((a_1,b_0,\cdots,b_d)\in\mathbb{A}^1\times\mathbb{A}^{d+1}\).
	In particular, the kernel contains \((a_1x_i,x_i)\) for \(i=0,\cdots,d\).

	If both \(e\) and \(h\) are invertible, \(a_1\neq 0\) and the kernel may alternatively be generated by
	\[
		\Big(1,\frac{1}{a_1}-\sum_{i=0}^d\frac{b_i}{a_1^2}x_i\Big).
	\]
	Fix \((u_0:u_1:\cdots:u_d)\in\mathbb{P}^d\) and set \(b_i=\frac{u_i}{t}\) for some \(t\neq 0\).
	Then the \(d+3\) generators of the kernel take the form
	\[
		\Big(a_1t+\sum_{i=0}^du_ix_i,t\Big),(a_1x_0,x_0),\cdots, (a_1x_d,x_d), \Big(a_1^2t,a_1t+\sum_{i=0}^du_ix_i\Big).
	\]
	Letting \(t\to 0\), these converge to
	\[
		\Big(\sum_{i=0}^du_ix_i,0\Big),(a_1x_0,x_0),\cdots, (a_1x_d,x_d), \Big(0,\sum_{i=0}^du_ix_i\Big)
	\]
	which generate the kernel described in \zcref{lem:structure_of_gamma} for \((u:v)=(a_1:1)\in\mathbb{P}^1\).
	This implies the limit of the point in \(M_f\) along the direction \((u_0:u_1:\cdots:u_d)\) is a point in \(\Gamma\).

	Next, assume \(e\) is invertible but \(h\) is not.
	Then \(a_1=0\) and the kernel contains
	\[
		\Big(\sum_{i=0}^d{b_ix_i},1\Big),(0,x_0),\cdots,(0,x_d).
	\]
	Again fix \((u_0:\cdots:u_d)\in\mathbb{P}^d\) and set \(b_i=\frac{u_i}{t}\) for \(t\neq 0\).
	Multiplying the first element by \(t\) gives \((\sum_{i=0}^d{u_ix_i},t)\).
	Letting \(t\to 0\), these \(d+2\) elements converge to
	\[
		\Big(\sum_{i=0}^d{u_ix_i},0\Big),(0,x_0),\cdots, (0,x_d),
	\]
	which generate the kernel appearing in \zcref{lem:structure_of_gamma} for \((u:v)=(0:1)\in\mathbb{P}^1\).

	The case where \(h\) is invertible and \(e\) is not is symmetric and yields limits with kernels
	of the type in \zcref{lem:structure_of_gamma} corresponding to \((u:v)=(1:0)\in\mathbb{P}^1\).
	These are exactly all the possible limits of points of \(M_f\), and they form the locus \(\Gamma\)
	Hence \(\overline{M_f}\setminus M_f=\Gamma\).
\end{proof}

\medskip

\subsection{\(d=1\)}
We first treat the case \(d=1\) and show that the moduli space of sheaf stable pairs
\(M_Z(2,2C,-2)\) is the projectivisation of a vector bundle of rank \(d+2=3\).

\begin{theorem}\label{thm:moduli_space_d1}
	Let \(C\) be a \((-1)\)-curve on a smooth projective surface \(Z\).
	Then the moduli space \(M_Z(2,2C,-2)\) of sheaf stable pairs with fixed Chern character
	is isomorphic to \(\mathbb{P}_{\mathbb{P}^1}(\mathcal{O}(2)^{\oplus 2}\oplus\mathcal{O})\).
\end{theorem}

To prove \zcref{thm:moduli_space_d1}, we first analyse the locus \(\Gamma\) from \zcref{def:limiting_case}.
In the case \(d=1\),
this is equivalent to parametrize all quotients
\[
	R^2\to k\oplus k[x,y]/(x,y)^2+(ux+vy),
\]
for all \((u:v)\in \mathbb{P}^1\), where \(R\) is the Artinian ring \(k[x,y]/(x,y)^2\).
Using \zcref{prop:theta_d_and_quot} we also have \(\Theta_1\simeq\Quot(\mathcal{O}_{C}^2,1)=C\times\mathbb{P}^1\),
and therefore \(\Gamma=\Theta_1\) (\zcref{lem:structure_of_gamma}).

The free \(R\)-module \(R^2\) may be viewed as a \(k\)-vector space of dimension \(6\),
with basis
\[
	(1,0), (0,1), (x,0), (0,x), (y,0), (0,y).
\]
We will realise \(M_Z(2,2C,-2)\) as a subvariety of the Grassmannian \(\Gr(3,6)\)
and study it in a neighbourhood of such a point.

\begin{lemma}\label{lem:smooth_oc_ocp}
	The moduli space \(M_Z(2,2C,-2)\) is smooth at every point of \(\Gamma\).
\end{lemma}

\begin{proof}
	Since \(M_Z(2,2C,-2)\simeq\Theta\) and \(\Theta\) parametrises certain quotients
	\[
		R^2\longrightarrow L,
	\]
	with \(L\) an \(R\)-module of dimension \(3\),
	the space \(M_Z(2,2C,-2)\) may be viewed as a subscheme of \(\Gr(3,6)\).

	We now write local equations for \(M_Z(2,2C,-2)\) inside the Grassmannian \(\Gr(3,6)\).
	Fix \((u:v)=(0:1)\) and consider the quotient \(R^2\to k\oplus k[x]/(x^2)\).
	The argument for general \((u:v)\in\mathbb{P}^1\) is identical.

	Let \(M_0\subseteq R^2\) be the submodule generated by the vectors \((x,0),(y,0),(0,y)\).
	Then \(R^2/M_0\simeq k\oplus k[x]/(x^2)\) as \(R\)-modules, and
	as a \(k\)-vector space it is spanned by the vectors \((1,0),(0,1),(0,x)\).
	Denote by \(w\) the corresponding point of \(M_Z(2,2C,-2)\),
	i.e., the quotient \(R^2\to k\oplus k[x]/(x^2)\) with kernel \(M_0\).

	Then the points in \(\Gr(3,6)\) around \(w\) represent the quotient \(R^2\to R^2/M\),
	where \(M\) is generated by the three rows \(R_1\), \(R_2\) and \(R_3\) of the matrix
	\begin{equation}
		\begin{pmatrix}
			a & 1 & 0 & b & c & 0 \\
			d & 0 & 1 & e & f & 0 \\
			g & 0 & 0 & h & i & 1
		\end{pmatrix},
	\end{equation}
	with \(a,b,\dots,i\) serving as local coordinates on \(\Gr(3,6)\) around \(w\).

	A necessary condition for such a quotient to define a point lies in \(M_Z(2,2C,-2)\)
	is that \(R^2/M\) is an \(R\)-module, i.e., when the subspace \(M\) is \((x,y)\)-invariant.
	Since \(x(a+x,b+cx)=(ax,bx)\), we have \(xR_1=aR_1\).
	Similarly, \(yR_1=aR_2+bR_3\), \(xR_2=dR_1\), \(yR_2=dR_2+eR_3\), \(xR_3=gR_1\) and \(yR_3=gR_2+hR_3\).
	These conditions translate into the equations
	\[
		\begin{cases}
			a=b=0,        \\
			e = cd,       \\
			d^2 + eg = 0, \\
			de + eh = 0,  \\
			df + ei = 0,  \\
			h = cg,       \\
			dg + gh = 0,  \\
			eg + h^2 = 0, \\
			fg + hi = 0.
		\end{cases}
	\]
	Simplify the equations, we get
	\[
		\begin{cases}
			(d+cg)d=0, \\
			(f+ci)d=0, \\
			(d+cg)g=0, \\
			(f+ci)g=0.
		\end{cases}
	\]
	Thus the locus is the union of two smooth components \(V(d,g)\) and \(V(d+cg,f+ci)\).
	The family \(M_f\) lies in the component \(V(d+cg,f+ci)\),
	hence locally around the point \(w\),
	the equations for \(M_Z(2,2C,-2)\) are given by \(d+cg=f+ci=0\).
	These two equations cut out a smooth subvariety, so
	\(M_Z(2,2C,-2)\) is smooth at \(w\).
\end{proof}

%

\begin{proof}[Proof of \zcref{thm:moduli_space_d1}]
	By \zcref{prop:M_f'_and_quot}, \(M_Z(2,2C,-2)=M_f\sqcup\Quot(\mathcal{O}_C^2,1)\).
	As observed above, \(\Gamma\) coincides with \(\Theta_1=\Quot(\mathcal{O}_{C}^2,1)\),
	and by \zcref{lem:limit_of_gamma} it parametrises precisely the limit points of \(M_f\).
	Thus \(M_Z(2,2C,-2)\) is obtained from \(M_f\) by adding the boundary \(\Gamma\).
	By \zcref{lem:quotient_to_o2c}, the locus \(M_f\) is smooth.
	By \zcref{lem:smooth_oc_ocp}, \(M_Z(2,2C,-2)\) is smooth around \(\Gamma\).
	Hence \(M_Z(2,2C,-2)\) is a smooth compactification of \(M_f\).

	Moreover, \zcref{lem:quotient_to_o2c} identifies \(M_f\) with the total space of \(\mathcal{O}(-2)^{\oplus 2}\),
	i.e., a \(\mathbb{A}^2\)-bundle over \(\mathbb{P}^1\).
	The projectivisation \(\mathbb{P}_{\mathbb{P}^1}(\mathcal{O}(2)^{\oplus 2}\oplus\mathcal{O})\)
	is also a smooth compactification of \(M_f\) by adding the divisor \(\mathbb{P}^1\times\mathbb{P}^1\) at infinity.
	Since the smooth compactification is unique by \cite[Proposition~(8.4.2)]{grothendieck1961elements},
	we get \(M_Z(2,2C,-2)=\mathbb{P}_{\mathbb{P}^1}(\mathcal{O}(2)^{\oplus 2}\oplus\mathcal{O})\).
\end{proof}

\subsection{\(d\geq 2\)}

Assume the curve \(C\) on the smooth surface \(Z\) satisfies \(C^2=-d\) with \(d\geq 2\).
Recall from \zcref{thm:M_f_theta_p} that \(\Theta_d\) is isomorphic to \(\Quot(\mathcal{O}_{C}^2,d)\).
Since \(\dim\Quot(\mathcal{O}_{C}^2,d)=2d\),
\(\Gamma\simeq\mathbb{P}^1\times\mathbb{P}^d\subseteq\Theta_d\) is a proper subset.
By embedding into \(\Gr(d+2,2(d+2))\), we have the following local description.

\begin{lemma}\label{lem:neibor_Gamma_2}
	Locally around each point of \(\Gamma\), the moduli space \(M_Z(2,2C,-2d)\) has two smooth components.
	One component with dimension \(d+2\) is in \(M_f\),
	and another component with dimension \(2d\) is in \(\Quot(\mathcal{O}_C^2,d)\).
\end{lemma}

\begin{proof}
	We only do the calculation for \(d=2\),
	the general case is similar but more complicated.

	In our setting, \(\Gamma\) parametrises all quotients
	\[
		R^2 \longrightarrow k\oplus k[x,y,z]/(x,y,z)^2+(a_1x+a_2y+a_3z)
	\]
	for all \((a_1:a_2:a_3)\in \mathbb{P}^2\),
	where \(R\) is the Artinian ring \(k[x,y,z]/(x,y,z)^2\).
	Let \(M_0\subseteq \mathbb{R}^2\) be the submodule generated by the vectors
	\[
		(x,0),(y,0),(z,0)(0,z).
	\]
	Then \(R/M_0\simeq k\oplus k[x,y]/(x,y)^2\) as \(R\)-modules,
	and as a \(k\)-vector space it is spanned by the vectors \((1,0),(0,1),(0,x),(0,y)\).
	Denote by \(w\) the corresponding point of \(M_Z(2,2C,-4)\),
	i.e., the quotient \(R^2\to k\oplus k[x,y]/(x,y)^2\) with kernel \(M_0\).

	Then the points in \(\Gr(4,8)\) around \(w\) represents the quotients \(R^2\to R^2/M\),
	where \(M\) is generated by the four rows \(R_1\), \(R_2\), \(R_3\) and \(R_4\) of the matrix
	\[
		\begin{pmatrix}
			a & 1 & 0 & 0 & b & c & d & 0 \\
			e & 0 & 1 & 0 & f & g & h & 0 \\
			i & 0 & 0 & 1 & j & k & l & 0 \\
			m & 0 & 0 & 0 & n & p & q & 1
		\end{pmatrix}
	\]
	where the entries \(a,\dots,l\) are local parameters.

	As before, for the quotient corresponding to a point of \(M_Z(2,2C,-4)\),
	the subspace \(M\) is \((x,y,z)\)-invariant.
	Then
	\[
		xR_1=(0,a,0,0,0,b,0,0)=aR_1=(a^2,a,0,0,ab,ac,ad,0),
	\]
	and we have \(a=0, b=ac=0\).
	From
	\[
		yR_2=(0,0,e,0,0,f,0,0)=eR_2=(e^2,e,0,0,ef,eg,eh,0),
	\]
	we get \(e=0, f=eg=0\).

	For simplicity, we rewrite the matrix as follows.
	\[
		\begin{pmatrix}
			0 & 1 & 0 & 0 & 0 & a & b & 0 \\
			0 & 0 & 1 & 0 & 0 & c & d & 0 \\
			e & 0 & 0 & 1 & f & g & h & 0 \\
			i & 0 & 0 & 0 & j & k & l & 1
		\end{pmatrix}
	\]
	Then
	\[
		x,y,zR_1=0=x,y,zR_2.
	\]
	Moreover, we have
	\begin{align*}
		& xR_3=(0,e,0,0,0,f,0,0)=eR_1=(0,e,0,0,0,ae,be,0), \\
		& xR_4=(0,i,0,0,0,j,0,0)=iR_1=(0,0,i,0,0,ai,bi,0), \\
		& yR_3=(0,0,e,0,0,0,f,0)=eR_2=(0,0,e,0,0,ce,de,0), \\
		& yR_4=(0,0,i,0,0,0,j,0)=iR_2=(0,0,i,0,0,ci,di,0).
	\end{align*}
	This implies
	\[
		f = ae = de, j = ai = di, 0 = be = bi = ce = ci,
	\]
	Similarly for \(z\), we get
	\[
		\begin{cases}
			e^2+fi = 0, \\
			ef+fj = 0, \\
			eg+fk = 0, \\
			eh+fl = 0,
		\end{cases}
	\]
	and
	\[
		\begin{cases}
			ei+ij = 0, \\
			fi+j^2 = 0, \\
			gi+jk = 0, \\
			hi+jl = 0.
		\end{cases}
	\]

	In the case \(b\neq 0\),
	one has \(e=i=0\), and hence \(f=j=0\).
	This is not the locus of \(\overline{M_f}\),
	since at least one of \(e,f,i,j\) is not identically zero.

	So we may assume that \(b=0\) and similarly \(c=0\).
	If \(a\neq d\) then \(e=i=0\) and hence \(f=j=0\),
	which is also not the case.
	Thus \(a=d\).
	Then if \(e=0\), one has \(f=j=0\).
	If \(i\neq 0\) then \(a=d=0\) and \(k,l\) are free.
	Similar for \(i=0\).

	Now we may also assume that \(e\neq 0\) and \(i\neq 0\).
	Then \(a=d\) and
	\[
		\begin{cases}
			f=ae, \\
			j=ai,
		\end{cases}
		\quad \text{and} \qquad
		\begin{cases}
			e+ai=0, \\
			g+ak=0, \\
			h+al=0.
		\end{cases}
	\]
	This is a 4-dimensional component with free variables \(a,i,k,l\).

	Therefore, locally the component \(\overline{M_f}\) is defined via
	\(V(b,c,a-d,e-ai,f+a^2i,g-ak,h-al,j-ai)\), which is smooth of dimension \(4\).
\end{proof}

Now we can prove the main theorem in this section.

\begin{theorem}\label{thm:main_negative_curve}
	Let \(C\) be a smooth rational curve with self-intersection \(-d\) on \(Z\) where \(d\geq 2\).
	Then the moduli space \(M_Z(2,2C,-2d)\) of
	sheaf stable pairs with the fixed Chen character
	has two smooth irreducible components,
	one is isomorphic to \(M_1=\mathbb{P}_{\mathbb{P}^1}(\mathcal{O}(2)^{\oplus d+1}\oplus\mathcal{O})\),
	the other is isomorphic to \(M_2=\Quot(\mathcal{O}_{C}^2,d)\);
	they intersects along \(\mathbb{P}^d\times\mathbb{P}^1\),
	which is the section at infinity of the \(\mathbb{P}^{d+1}\)-bundle \(M_1\to\mathbb{P}^1\).
\end{theorem}

\begin{proof}
	By \zcref{prop:M_f'_and_quot}, we have
	\(M_Z(2,2C,-2d)=M_f\sqcup\Quot(\mathcal{O}_C^2, d)\).
	We denote the closure of \(M_f\) in \(M_Z(2,2C,-2d)\) by \(\overline{M_f}\) which gives a
	compactification of \(M_f\).
	By \zcref{lem:limit_of_gamma}, we have \(\overline{M_f}=M_f\sqcup\Gamma\).
	By \zcref{lem:quotient_to_o2c}, \(M_f\) is a smooth variety.
	Also by \zcref{lem:neibor_Gamma_2}, for any point in \(\Gamma\),
	it has a smooth neighbourhood in \(\overline{M_f}\).
	Hence \(\overline{M_f}=M_f\sqcup\Gamma\) is a smooth compactification of \(M_f\).

	By \zcref{lem:quotient_to_o2c}, \(M_f\) is the total space of \(\mathcal{O}(-2)^{\oplus d+1}\),
	i.e., a \(\mathbb{A}^{d+1}\)-bundle over \(\mathbb{P}^1\).
	The projectivisation
	\(\mathbb{P}_{\mathbb{P}^1}(\mathcal{O}(2)^{\oplus {d+1}}\oplus \mathcal{O})\) is also a smooth compactification
	of \(M_f\) by gluing \(\mathbb{P}^{d}\times \mathbb{P}^1\) at infinity.
	Since the smooth compactification is unique by \cite[Proposition~8.4.2]{grothendieck1961elements},
	we get \(\overline{M_f}=\mathbb{P}_{\mathbb{P}^1}(\mathcal{O}(2)^{\oplus d+1}\oplus \mathcal{O})\)
	and \(\Gamma=\mathbb{P}^{d}\times \mathbb{P}^1\).
	Now \(M_Z(2,2C,-2d)=\overline{M_f}\cup\Quot(\mathcal{O}_C^2,d)\) consists of two components,
	the intersection of the two components \(\overline{M_f}\cap\Quot(\mathcal{O}_C^2,d)=\Gamma\).
\end{proof}


\printbibliography

\end{document}